\documentclass[12pt, a4paper]{amsart}
\usepackage{amsmath}
\usepackage{geometry,amsthm,graphics,tabularx,amssymb,shapepar}
\usepackage{amscd}
\usepackage{graphicx}
\usepackage[all]{xypic}
\usepackage{hyperref}
\usepackage{color}
\usepackage{soul}
\usepackage{bm}
\usepackage{bbm}
\allowdisplaybreaks[1]

\newcommand{\fg}{\mathfrak g}

\newcommand{\C}{\mathbb{C}}

\newcommand{\N}{\mathbb N}
\newcommand{\Z}{\mathbb Z}
\newcommand{\rd}{\mathrm{d}}
\newcommand{\rD}{\mathrm{D}}

\newcommand{\rh}{\mathrm{h}}
\newcommand{\re}{\mathrm{e}}
\newcommand{\rf}{\mathrm{f}}

\newcommand{\ba}{\begin {eqnarray}}
\newcommand{\ea}{\end {eqnarray}}
\newcommand{\baa}{\begin {eqnarray*}}
\newcommand{\eaa}{\end {eqnarray*}}
\newcommand{\be}{\begin {equation}}
\newcommand{\ee}{\end {equation}}
\newcommand{\bee}{\begin {equation*}}
\newcommand{\eee}{\end {equation*}}

\newcommand{\U}{\mathcal{U}}

\def \<{{\langle}}
\def \>{{\rangle}}

\def \({ \left( }
\def \){ \right) }
\theoremstyle{Theorem}

\theoremstyle{Theorem}

\newtheorem{thm}{Theorem}[section]
\newtheorem{corollary}[thm]{Corollary}
\newtheorem{lemma}[thm]{Lemma}
\newtheorem{proposition}[thm]{Proposition}

\newtheorem{remark}[thm]{Remark}

\numberwithin{equation}{section}

\begin{document}

\title[Whittaker modules over loop Virasoro algebra]{Whittaker modules over the loop Virasoro algebra}

\author{Zhiqiang Li$^1$}
\address{School of Mathematics and Statistics, Xinyang Normal University,
 Xinyang, China 464000} \email{lzq06031212@xynu.edu.cn}
 \thanks{$^1$Partially Supported by Nanhu Scholars Program of
XYNU(No. 2019021).}

 \author{Shaobin Tan$^2$}
 \address{School of Mathematical Sciences, Xiamen University,
 Xiamen, China 361005} \email{tans@xmu.edu.cn}
  \thanks{$^2$Partially supported by NSF of China (No. 12131018).}

\author{Qing Wang$^3$}
\address{School of Mathematical Sciences, Xiamen University,
 Xiamen, China 361005} \email{qingwang@xmu.edu.cn}
 \thanks{$^3$Partially Supported by China NSF grants (No. 12571033).}

\subjclass[2010]{17B67, 17B10}
\keywords{Whittaker module, loop Virasoro algebra, affine Lie algebra $\widehat{\mathfrak{sl}_{2}}$}

\maketitle

\begin{abstract}
In this paper, we first study two classes of Whittaker modules over the loop Witt algebra
$\fg:=\mathcal{W}\otimes\mathcal{A}$, where $\mathcal{W}=\text{Der}(\C[t])$, $\mathcal{A}=\C[t,t^{-1}]$. The necessary and sufficient conditions for these Whittaker modules being simple are determined.
Furthermore, we study a family of Whittaker modules
over the loop Virasoro algebra $\mathfrak{L}:=Vir\otimes\mathcal{A}$, where $Vir$ is the Virasoro algebra. The irreducibility criterion for these Whittaker modules are obtained. As an application,
we give the irreducibility criterion for universal Whittaker modules of the affine Lie algebra
$\widehat{\mathfrak{sl}_{2}}$.

\end{abstract}

\section{Introduction}
Whittaker modules are important objects in the study of
representation theory of Lie algebras. Arnal and Pinzcon first defined the Whittaker module
for $\mathfrak{sl}_{2}$ in \cite{AP}.
Then Whittaker modules for complex semisimple Lie algebras were defined by Kostant in \cite{Kos}. Specially,
let $\mathfrak{g}=\mathfrak{n}_{-}\oplus\mathfrak{h}\oplus\mathfrak{n}_{+}$ be a complex semisimple Lie algebra
and $\eta:\mathfrak{n}_{+}\rightarrow\C$ be a Lie algebra homomorphism.
A $\mathfrak{g}$-module is called a Whittaker module
if $x-\eta(x)$ acts locally nilpotently for any $x\in\mathfrak{n}_{+}$.
When the Lie algebra homomorphism $\eta$ is non-singular, Kostant proved that simple Whittaker modules
are in one-to-one correspondence with maximal ideals of the center $Z(\mathfrak{g})$ of
the universal enveloping algebra $\U(\mathfrak{g})$. In \cite{Blo}, Block proved that simple
modules for $\mathfrak{sl}_2$ are two families:
one is a family of simple weight modules,
and the other is a family of non-weight modules including Whittaker modules.
%After that, McDowell considered a category of $\mathfrak{s}$-modules in \cite{McD} which contains the BBG category $\mathcal{O}$
% and those Whittaker modules where the Whittaker function on a nilpotent radical may be irregular (degenerate).
 %The simple objects in this category are constructed by inducing over a parabolic subalgebra $\mathfrak{p}$ of $\mathfrak{s}$ from a simple Whittaker module (in Kostant's sense) or
% from a highest weight module for the reductive Levi factor of $\mathfrak{p}$ (when the Whittaker function is zero).
%Whittaker modules for finite dimensional
%simple complex Lie algebra are also related to parabolic induction
%and BGG category $\mathcal{O}$ (\cite{Bac,KM}).
%For a simple $\mathfrak{p}$-module with a minimal possible annihilator,
%where $\mathfrak{p}$ is a parabolic subalgebra of the complex semisimple Lie algebra $\mathcal{L}$,
%Khomenko and Mazorchuk proved that the induced $\mathcal{L}$-module from $V$
%corresponds to proper standard modules in some parabolic generalization of the BGG category $\mathcal{O}$ of $\mathcal{L}$.
Whittaker modules have been extensively studied for various algebraic structure, for example,
quantum groups \cite{Ond,S}, (generalized) Weyl algebras \cite{BO},
affine Lie algebras \cite{ALZ,Chr,CGLW,CJ,GL,M}, twisted Heisenberg-Virasoro algebras of rank two \cite{TWX},
Euclidean Lie algebra $\mathfrak{e}(3)$ \cite{CSZ},
the Virasoro algebra \cite{FJK,LGZ,OW,Y}, classical Lie superalgebras in \cite{Che}
and so on. Stimulated by these works, Batra and Mazorchuk introduced a Whittaker pair and then they defined
Whittaker modules based on the Whittaker pairs in \cite{BM} (see also \cite{MZ2}).
In this paper, we consider a family of Whittaker pairs for the loop Virasoro algebra $\mathfrak{L}$
and the loop Witt algebra $\mathfrak{g}$, then we study their corresponding Whittaker modules, respectively.

Let $\mathcal{A}=\C[t,t^{-1}]$  be the Laurent polynomial
algebra and let $\mathcal{A}^{+}=\C[t]$ be a subalgebra of it.
Witt Lie algebra $\mathrm{Der}(\mathcal{A})$ (resp. $\mathcal{W}=\mathrm{Der}(\mathcal{A}^{+})$)
is the derivative Lie algebra of $\mathcal{A}$ (resp. $\mathcal{A}^{+}$).
And the Virasoro algebra $Vir$ is the universal central extension of $\mathrm{Der}(\mathcal{A})$ with the basis
$\{\rd_{n}:=t^{n+1}\frac{\rd}{\rd t}, \bm{c}\mid n\in\Z\}$. The representation theory of $Vir$ has been
fully studied in \cite{Mat,MaP,MZ1,MZ2} and references therein. In this paper, we consider
the loop Virasoro algebra $\mathfrak{L}=Vir\otimes\mathcal{A}$
(resp. loop Witt algebra $\mathfrak{g}=\mathrm{Der}(\mathcal{A}^{+})\otimes\mathcal{A}$).
%As universal central extension of loop algebras of finite dimensional simple Lie
%algebras, simply-laced affine Kac-Moody algebras have a very different representation
%theory and have more applications than finite dimensional simple Lie algebras
%to other fields \cite{Kap,MP}.
%The goal of the current paper is to consider
%Whittaker modules over loop Virasoro algebra and loop Witt algebras and, which are Lie algebras of the form $\mathfrak{L}=Vir\otimes\mathcal{A}$ and $\mathfrak{g}=\mathcal{W}\otimes\mathcal{A}$, respectively.
The structure theory of $\mathfrak{L}$ has been studied
in \cite{OR} from the point of view of the Kadomtsev-Petviashvili equation,
while the representation theory of $\mathfrak{L}$ has been studied in \cite{GLZ,LG}. Specifically,
the $\Z$-graded Harish-Chandra modules over $\mathfrak{L}$ were classified in \cite{GLZ}.
And Liu-Guo studied a kind of Whittaker module over $\mathfrak{L}$ in \cite{LG}.
They use $\Z^{2}$-gradation of $\mathfrak{L}$ to define Whittaker modules under certain total order of $\Z^{2}$.
Specifically, for any total order $``\preceq"$ on $\Z^{2}$, set
$\mathfrak{L}_{+}=\text{Span}_{\C}\{\rd_{i}\otimes t^{k},
\bm{c}\otimes t^{l}\mid (i,k)\succ(0,0), (0,l)\succ(0,0)\}$. They proved that
$(\mathfrak{L},\mathfrak{L}_{+})$ is a Whittaker pair if and only if
the order $``\preceq"$ is discrete, and with the discrete order $``\preceq"$,
they defined a universal Whittaker module for $\mathfrak{L}$ and determined all Whittaker vectors.
%The Whittaker modules over $\mathfrak{L}$ were studied in \cite{LG}.
%This algebra is closely related to the algebra $W(2,2)$ \cite{ZD} and the truncated Virasoro algebras \cite{GL,W2}.
%Note that $\mathrm{Der}(\mathcal{A})\otimes\mathcal{A}$, which is the loop centerless Virasoro algebra,
%is a subalgebra of $\text{Der}(\C[t_{1}^{\pm1},t_{2}^{\pm1}])$ the derivative algebra of $\C[t_{1}^{\pm1},t_{2}^{\pm1}]$.
%The representation theory of $\mathfrak{L}$ plays an important role in the representation theory of  $\text{Der}(\C[t_{1}^{\pm1},t_{2}^{\pm1}])$
%(see \cite{BMZ} for details). Thus, we study other type of Whittaker modules for $\mathfrak{L}$.
%Other reasons for us to study the loop-Witt algebra $\mathfrak{L}$ is that the representations of $\mathrm{Der}(\mathcal{A})\otimes\mathcal{A}$
%plays an important role in the theory of representation of $\text{Der}(\C[t_{1}^{\pm1},t_{2}^{\pm1}])$
%(see \cite{BMZ} for details).
While the Whittaker modules for $\mathfrak{L}$ we studied in this paper are different from those in \cite{LG},
we use the $\Z$-gradation of $\mathfrak{L}$
which is induced from the natural $\Z$-gradation of $Vir$ to give Whittaker pairs and
then to study the corresponding Whittaker modules.

The paper is organized as follows. In Section $2$, we introduce two kinds of Whittaker pairs
$(\fg,\fg_{\geq N})$ and $(\fg,\fg_{-1})$ for the loop Witt algebra $\fg$
with a positive integer $N$. Then we define two families of universal
Whittaker modules $W(\fg_{\geq N},\phi)$ (resp. $W(\fg_{-1},\phi)$) over $\fg$
associated to $(\fg,\fg_{\geq N})$ (resp. $(\fg,\fg_{-1})$), where
$\phi:\fg_{\geq N}\rightarrow \C$ (resp. $\phi:\fg_{-1}\rightarrow \C$)
is a Lie algebra homomorphism.
In Section $3$ and Section $4$, the necessary and sufficient conditions
for $W(\fg_{\geq N},\phi)$ and $W(\fg_{-1},\phi)$ being simple are obtained.
In Section $5$, we first give an irreducibility criterion for a family of
Whittaker modules over the loop Virasoro algebra $\mathfrak{L}$,
then as an application, we character the simplicity of the universal Whittaker $\widehat{\mathfrak{sl}_{2}}$-modules.
We want to mention that the methods we used to study the Whittaker modules for $\fg$ in this paper have generality,
it may extend to other algebras like $\mathfrak{L}$ and $\widehat{\mathfrak{sl}_{2}}$.

All the Lie algebras considered in this paper are over the field $\C$ of complex numbers.
We denote the sets of integers, nonnegative integers, positive ingegers, complex numbers,
and nonzero complex numbers by $\Z$, $\N$, $\Z_+$, $\C$, and $\C^\times$, respectively.
For a Lie algebra $L$, we will use the notation $\U(L)$ for the universal enveloping algebra of $L$.

\section{preliminaries}
\subsection{Whittaker modules}
In this subsection, we recall some notions and results about Whittaker modules for later use.

From \cite{BM}, let $\mathcal{L}$ be a Lie algebra and $\mathfrak{a}$ a Lie subalgebra of $\mathcal{L}$.
The ordered pair $(\mathcal{L},\mathfrak{a})$ is called a $\textsl{Whittaker pair}$ if
$\mathfrak{a}$ acts on the $\mathfrak{a}$-module $\mathcal{L}/\mathfrak{a}$ locally nilpotently and  $\cap_{i=0}^{\infty}\mathfrak{a}_{i}=0$, where
\[\mathfrak{a}_{i+1}=[\mathfrak{a},\mathfrak{a}_{i}],\quad i=0,1,\dots,\ \mathfrak{a}_{0}=\mathfrak{a}.\]
$\phi$ is called a $\textsl{Whittaker function}$ if $\phi:\mathfrak{a}\rightarrow\C$ is a Lie algebra homomorphism.
Let $V$ be a $\mathcal{L}$-module.
A vector $v\in V$ is called a $\textsl{Whittaker vector of type}$ $\phi$ if $xv=\phi(x)v$ for all $x\in\mathfrak{a}$.
The module $V$ is called a $\textsl{Whittaker module of type}$ $\phi$
if $V$ is generated by a nonzero Whittaker vector of type $\phi$. For a Whittaker $\mathcal{L}$-module $V$ of type
$\phi$, we denote by $V_{\phi}$ the set of Whittaker vectors of type $\phi$, i.e.,
\[V_{\phi}=\{v\in V\mid xv=\phi(x)v,\ x\in\mathfrak{a}\}.\]
Let $\mathfrak{a}^{(\phi)}=\{x-\phi(x)\mid x\in \mathfrak{a}\}$,
which is a Lie subalgebra of the universal enveloping algebra $\U(\mathfrak{a})$.

We have the following results (see \cite{ALZ}, \cite{CJ}).
\begin{lemma}\label{ALZ}
Let $\mathcal{L}$ be a Lie algebra
and $\mathfrak{a}$ be a Lie subalgebra of $\mathcal{L}$ such that $\mathfrak{a}$ acts on the $\mathfrak{a}$-module $\mathcal{L}/\mathfrak{a}$ locally nilpotently,
and let $V$ be a Whittaker $\mathcal{L}$-module of type $\phi$. Then the following hold.\\
$(1)$ $\mathfrak{a}^{(\phi)}$ acts locally nilpotent on $V$.
In particular, $x-\phi(x)$ acts locally nilpotently on $V$ for any $x\in\mathfrak{a}$;\\
$(2)$ Any non-zero submodule of $V$ contains a non-zero Whittaker vector of type $\phi$;\\
$(3)$ If the vector space of Whittaker vector of $V$ is $1$-dimensional, then $V$ is simple.
\end{lemma}

\begin{lemma}\label{CJ}
Let $\mathcal{L}$ be a Lie algebra and $\mathfrak{a}$
be a Lie subalgebra of $\mathcal{L}$ such that $\mathfrak{a}$ acts
on the $\mathfrak{a}$-module $\mathcal{L}/\mathfrak{a}$ locally nilpotently,
and let $V$ be a Whittaker $\mathcal{L}$-module of type $\phi$. Then Whittaker vectors in $V$ are all of type $\phi$.
\end{lemma}

\subsection{Loop Witt algebra $\fg$ and Whittaker $\fg$-module}
Let $\mathcal{A}=\C[t,t^{-1}]$ and let $\mathcal{W}=\text{Der}(\C[t])$ be
the $\textsl{Witt Lie algebra}$. In this subsection, we review some basics on the $\textsl{loop Witt algebra}$
$\fg:=\mathcal{W}\otimes\mathcal{A}$ and consider two families of Whittaker modules over $\fg$.

It is well known that \[\mathcal{W}=\underset{i\in\Z,\ i\geq-1 }{\bigoplus}\C \rd_{i},\]
where $\rd_{i}=t^{i+1}\frac{\rd}{\rd t}$. The Lie brackets are given by
\[[\rd_{n},\rd_{m}]=(m-n)\rd_{n+m},\quad n,m\geq-1.\]
And the Lie brackets in $\fg$ are given by
\begin{align}\label{Lie-bracket}[\rd_{n}\otimes t^{k},\rd_{m}\otimes t^{l}]=(m-n)\rd_{n+m}\otimes t^{k+l},\quad n,m\geq-1,\ k,l\in\Z.\end{align}
There exists a $\Z$-grading on $\fg$ given by
\[\fg=\bigoplus_{n\in\Z}\fg_{n},\]
where $\fg_{n}=\rd_{n}\otimes\mathcal{A}$ for $n\geq-1$ and $\fg_{n}=0$ for $n<-1$.
In the rest of this paper, let $N\in\Z_{+}$ be fixed. Denote
\[\fg_{\geq N}=\bigoplus_{n\geq N}\rd_{n}\otimes\mathcal{A}.\]
It is clear that $\fg_{\geq N}$ is a subalgebra of $\fg$ and $(\fg,\fg_{\geq N})$ is a Whittaker pair.

Let $\phi:\fg_{\geq N}\rightarrow\C$ be a Whittaker function. For any $n\geq N$,
set $\phi_{n}\in\mathcal{A}^*$ such that
\begin{align}\label{psi-n}\phi_{n}(t^{r})=\phi(\rd_{n}\otimes t^{r})\quad \text{for all}\quad r\in\Z.\end{align}
Since $\phi$ is a Lie algebra homomorphism, we know that
\[\phi_{i}=0\quad\text{for all}\quad i\geq2N+1.\] Define the Whittaker module
$W(\fg_{\geq N},\phi)$ over $\fg$ as follows:
\begin{align}\label{Universal-Whittaker-mod}
W(\fg_{\geq N},\phi)=\U(\fg)\otimes_{\U(\fg_{\geq N})}\C v_{\phi},
\end{align}
where $\C v_{\phi}$ is the one dimensional $\fg_{\geq N}$-module given by
$x.v_{\phi}=\phi(x)v_{\phi}$ for any $x\in\fg_{\geq N}$. It is clear that for any Whittaker
$\fg$-module $V$ generated by a Whittaker vector $w$ of type $\phi$, there exists a surjective $\fg$-module
homomorphism \[\Phi: W(\fg_{\geq N},\phi)\rightarrow V\ \text{such that}\ \Phi(v_{\phi})=w.\]
Thus we call $W(\fg_{\geq N},\phi)$
$\textsl{the universal Whittaker module}$ associated to $(\fg,\fg_{\geq N})$ $\textsl{of type}$ $\phi$.

We know that $(\fg,\fg_{-1})$ is also a Whittaker pair.
Note that $\fg_{-1}$ is a commutative subalgebra of $\fg$, so for any $\phi\in\fg_{-1}^*$,
we can define the $\textsl{the universal Whittaker module}$
$W(\fg_{-1},\phi)=\U(\fg)\otimes_{\U(\fg_{-1})}\C v_{\phi}$ associated to
$(\fg,\fg_{-1})$ $\textsl{of type}$ $\phi$,
where $\C v_{\phi}$ is the one dimensional $\fg_{-1}$-module given by
$x.v_{\phi}=\phi(x)v_{\phi}$ for any $x\in\fg_{-1}$, similarly.

\section{Whittaker modules associated to $(\fg,\fg_{\geq N})$}
Let $\psi:\fg_{\geq N}\rightarrow\C$ be a Whittaker function.
This section is devoted to characterize the simplicity of
the universal Whittaker module $W(\fg_{\geq N},\psi)$. Firstly, we recall the exp-polynomial
function from \cite{BZ}.

A linear functional $\phi$ on $\mathcal{A}$
is called $\textsl{exp-polynomial}$ if it can be written as a finite sum
\[\phi(t^{n})=\sum_{k\in\N}\sum_{\lambda\in\C^{\times}}c_{k,\lambda}n^{k}\lambda^{n},\]
where $c_{k,\lambda}\in\C$. Denote by $\mathcal{E}$ the subspace of $\mathcal{A}^*$ consisting
of all exp-polynomial functionals on $\mathcal{A}$.
We define an $\mathcal{A}$-action on its dual space $\mathcal{A}^*$ by letting
\[(t^{m}.\phi)(t^n)=\phi(t^{n+m})\quad\text{for}\ m,n\in\Z,\ \phi\in\mathcal{A}^*.\]
This gives an $\mathcal{A}$-module structure on $\mathcal{A}^*$.
For every $\phi\in\mathcal{A}^*$, let \[\mathrm{Ann}(\phi)=\{h(t)\in\mathcal{A}\mid h(t).\phi=0\}\]
be the $\textsl{annihilator ideal}$ of $\mathcal{A}$ associated to $\phi$.
The following lemma is straightforward to prove and also can be referenced in \cite[Proposition 2.3]{W1}.
\begin{lemma}\label{lem:charCE}
For any $\phi\in\mathcal{A}^*$,
the annihilator ideal $\mathrm{Ann}(\phi)\neq 0$ if and only if $\phi\in\mathcal{E}$. Moreover,
$\mathcal{E}$ is an $\mathcal{A}$-submodule of $\mathcal{A}^*$.
\end{lemma}

\begin{lemma}\label{lem:non-CE}
For any $\phi\in\mathcal{A}^*\setminus\mathcal{E}$ and $f(t)\in\mathcal{A}$, we have $f(t).\phi\in\mathcal{E}$ if and
only if $f(t)=0$. Thus for any $\phi\in\mathcal{A}^*\setminus\mathcal{E}$ and non-zero element $f(t)\in\mathcal{A}$, \[f(t).\phi\in\mathcal{A}^*\setminus\mathcal{E}.\]
\end{lemma}
\begin{proof}If $f(t)=0$, it is obvious. Conversely, suppose $f(t).\phi\in\mathcal{E}$. This means that
there exists some non-zero polynomial $g(t)\in\mathrm{Ann}(f(t).\phi)$. That is $g(t)f(t)\in\mathrm{Ann}(\phi)$.
Thus $g(t)f(t)=0$, i.e., $f(t)=0$.
\end{proof}

Now we state the main result of this section, we prove it
by using Proposition \ref{psi-reducible} to Proposition \ref{EXP-2N-1}.
\begin{thm}\label{first-main-restlt}
Let $\psi:\fg_{\geq N}\rightarrow\C$ be a Whittaker function.
Then the universal Whittaker $\fg$-module $W(\fg_{\geq N},\psi)$ is simple if and only if
either $\psi_{2N-1}\notin\mathcal{E}$ or $\psi_{2N}\notin\mathcal{E}$.
\end{thm}
For convenience, we write $\rD(n,k)=\rd_{n}\otimes t^{k}$ for $n\geq-1$ and $k\in\Z$.
Thus, \eqref{Lie-bracket} can be rewritten as
\begin{align}
[\rD(n,k),\rD(m,l)]=(m-n)\rD(n+m,k+l).
\end{align}

\begin{proposition}\label{psi-reducible}
If $\psi_{2N-1},\psi_{2N}\in\mathcal{E}$, then $W(\fg_{\geq N},\psi)$ is reducible.
\end{proposition}
\begin{proof}
From Lemma \ref{lem:charCE}, there exists a polynomial $c(t)=\sum_{j=0}^{q}c_{j}t^{j}$
with degree $q\geq1$ such that $c(t)\in\mathrm{Ann}(\psi_{2N-1})\cap\mathrm{Ann}(\psi_{2N})$.
Consider the non-zero vector \[u=\sum_{j=0}^{q}c_{j}\rD(N-1,j)v_{\psi}\in W(\fg_{\geq N},\psi).\]
We have \begin{eqnarray*}
\big(\rD(i,k)-\psi_{i}(t^{k})\big).u&=&(N-1-i)\sum_{j=0}^{q}c_{j}\psi_{N-1+i}(t^{k+j})v_{\psi}\\
&=&(N-1-i)\Big(\big(c(t).\psi_{N-1+i}\big)(t^{k})\Big)v_{\psi}=0
\end{eqnarray*}
and
\[\big(\rD(i',k)-\psi_{i'}(t^{k})\big).u=0\]
for $i=N,N+1$, $i'\geq N+2$, and $k\in\Z$. This implies that \[\U(\fg)u=\U(\fg_{-1})\U(\fg_{0})\cdots\U(\fg_{N-1})u.\]
Thus, the non-zero submodule $\U(\fg)u$ of
$W(\fg_{\geq N},\psi)$ is a proper submodule. This completes the proof.
\end{proof}

Now we introduce some notations and the lexicographical total order for later use. Set
\begin{eqnarray*}&&B(\fg_{\geq N},\psi)=\{\rD(r_{1},k_{r_{1},1})\rD(r_{1},k_{r_{1},2})\cdots
\rD(r_{1},k_{r_{1},s_{1}})\rD(r_{2},k_{r_{2},1})\cdots
\rD(r_{2},k_{r_{2},s_{2}})\\&&
\cdots\rD(r_{n},k_{r_{n},1})\rD(r_{n},k_{r_{n},2})\cdots
\rD(r_{n},k_{r_{n},s_{n}})v_{\psi}\mid n\in\N,\ k_{r_{i},j}\in\Z,\ 1\leq j\leq s_{i},
\\&&\quad 1\leq i\leq n,\ -1\leq r_{n}<\cdots<r_{1}\leq N-1,\
k_{r_{i},s_{i}}\leq\cdots\leq k_{r_{i},2}\leq k_{r_{i},1}\}.\end{eqnarray*}
By Poincar-Birkhok-Witt theorem, we know that $B(\fg_{\geq N},\psi)$ forms a basis of $W(\fg_{\geq N},\psi)$.

For convenience, we make the convention
\[g(n)g(n-1)\cdots g(m):=\prod_{i=n}^{m}g(i)\]
for any function $g$ and integers $n,m$ with $n\geq m$.

For every element \begin{align}\label{typical-vector}u=
\prod_{i=N-1}^{-1}\prod_{j=1}^{r_{i}}\rD(i,k_{i,j})^{l_{i,j}}v_{\psi}\in B(\fg_{\geq N},\psi),\end{align}
where $l_{i,j}\in\N$, $k_{i,j}\in\Z$, $-1\leq i\leq N-1$, $1\leq j\leq r_{i}$, write
\[\text{lth}_{i}(u)=\sum_{j=1}^{r_{i}}l_{i,j},\quad \text{lth}(u)=\sum_{i=-1}^{N-1}\text{lth}_{i}(u),\]
\[\mathcal{D}(u)=\big(\underbrace{N-1,N-1,\dots,N-1}_{\mathrm{lth}_{N-1}(u)-\mathrm{times}},
\dots,\underbrace{-1,-1,\dots,-1}_{\mathrm{lth}_{-1}(u)-\mathrm{times}}\big),\]
\[\mathcal{D}_{\mathrm{set}}(u)=\{i\mid -1\leq i\leq N-1,\ \text{lth}_{i}(u)\geq1\},\]
$$\begin{aligned}\mathcal{T}_{N-1}(u)=\big(\underbrace{k_{N-1,1},k_{N-1,1},\dots,k_{N-1,1}}_{l_{N-1,1}-\mathrm{times}},\dots
\underbrace{k_{N-1,r_{N-1}},k_{N-1,r_{N-1}},\dots,k_{N-1,r_{N-1}}}_{l_{N-1,r_{N-1}}-\mathrm{times}}\big),\end{aligned}$$
\[\cdots\cdots\]
$$\begin{aligned}\mathcal{T}_{-1}(u)=\big(\underbrace{k_{-1,1},k_{-1,1},\dots,k_{-1,1}}_{l_{-1,1}-\mathrm{times}},\cdots,
\underbrace{k_{-1,r_{-1}},k_{-1,r_{-1}},\dots,k_{-1,r_{-1}}}_{l_{-1,r_{-1}}-\mathrm{times}}\big),\end{aligned}$$
\[\mathcal{T}(u)=\big(\mathcal{T}_{k_{N-1}}(u),\dots,\mathcal{T}_{-1}(u)\big),\]
\[\mathcal{T}_{i,\mathrm{set}}(u)=\{k_{i,j}\mid j=1,\dots,r_{i},\ l_{i,j}\geq1\}\quad \text{for}\quad -1\leq i\leq N-1,\]
\[\mathcal{T}_{\mathrm{set}}(u)=\{k_{i,j}\mid -1\leq i\leq N-1,\ j=1,\dots,r_{i},\ l_{i,j}\geq1\}.\]

For any $r\in\Z_+$, let ``$\succ$'' be the lexicographical total order on $\Z^{r}$.
Namely,
for \[\bm{a}=(a_1,a_2,\dots,a_{r}),\ \bm{b}=(b_1,b_2,\dots,b_{r})\in\Z^{r},\]
$\bm{a}\succ\bm{b}$ if and only if there exists $j\in\Z_+$ such that
\begin{equation}\label{total-order}a_j>b_j\ \text{and}\ a_{i}=b_{i},\ 1\leq i<j\leq r.\end{equation}

Then we define the principle total order ``$\succ$'' on $B(\fg_{\geq N},\psi)$ as follows:
for different $u,v\in B(\fg_{\geq N},\psi)$, set $u\succ v$ if and only if one of the following conditions is satisfied:
     \begin{itemize}
     \item $\text{lth}(u)>\text{lth}(v)$;
     \item $\text{lth}(u)=\text{lth}(v)$ and $\mathcal{D}(u)\succ\mathcal{D}(v)$ under the total on $\Z^{\text{lth}(u)}$;
     \item $\text{lth}(u)=\text{lth}(v)$, $\mathcal{D}(u)=\mathcal{D}(v)$,
     and $\mathcal{T}(u)\succ \mathcal{T}(v)$ under the total on $\Z^{\text{lth}(u)}$.
     \end{itemize}

We have the following lemma. The proof is straightforward.
\begin{lemma}\label{direct-calculate}
Let \begin{eqnarray*}&&\mu=\rD(r_{1},k_{r_{1},1})\rD(r_{1},k_{r_{1},2})\cdots
\rD(r_{1},k_{r_{1},s_{1}})\rD(r_{2},k_{r_{2},1})\cdots
\rD(r_{2},k_{r_{2},s_{2}})\\&&
\quad\cdots\cdots\rD(r_{n},k_{r_{n},1})\rD(r_{n},k_{r_{n},2})\cdots
\rD(r_{n},k_{r_{n},s_{n}})v_{\psi}\in B(\fg_{\geq N},\psi),\end{eqnarray*}
where $n\in\N$, $k_{r_{i},j}\in\Z$, $-1\leq r_{n}<\cdots<r_{1}\leq N-1$,
$k_{r_{i},s_{i}}\leq\cdots\leq k_{r_{i},2}\leq k_{r_{i},1}$, $1\leq i\leq n$, $1\leq j\leq s_{i}$.
Then for any $m\geq N$ and $k\in\Z$, we have
\[\big(\rD(m,k)-\psi_{m}(t^{k})\big).\mu
=\prod_{i=1}^{n}\prod_{j=1}^{s_{i}}\big[\rD(m,k),\rD(r_{i},k_{r_{i},j})\big]v_{\psi}.\]
Moreover, if we write
$\big(\rD(m,k)-\psi_{m}(t^{k})\big).\mu=\sum_{l=1}^{p}a_{l}v_{l}$,
where $p\in\N$, $a_{1},\dots,a_{p}\in\C^{\times}$, and $v_{1},\dots,v_{p}\in B(\fg_{\geq N},\psi)$ are distinct elements,
then \[\mathrm{lth}(v_{l})\leq \mathrm{lth}(\mu)\]
for $l=1,\dots,p$.
\end{lemma}

We need the following fact
\begin{proposition}\label{ele-lemma}
Let $u\in B(\fg_{\geq N},\psi)$ be as in \eqref{typical-vector}. For $k\in\Z$, write
\[\big(\rD(N,k)-\psi_{N}(t^{k})\big).u=\sum_{n=1}^{q}b_{n}v_{n},\]
where $q\in\N$, $b_{1},\dots,b_{q}\in\C^{\times}$,
and $v_{1},\dots,v_{q}\in B(\fg_{\geq N},\psi)$ are distinct elements.\\
(1) If $\mathrm{lth}_{-1}(u)=0$, then
\[\mathrm{lth}(v_{n})<\mathrm{lth}(u),\quad\mathcal{D}_{\mathrm{set}}(v_{n})\subseteq\mathcal{D}_{\mathrm{set}}(u),
\quad\text{and}\quad\mathcal{T}_{\mathrm{set}}(v_{n})\subseteq\mathcal{T}_{\mathrm{set}}(u)\]
for all $n=1,\dots,q$.\\
(2) If $\mathrm{lth}_{-1}(u)>0$, then $\big(\rD(N,k)-\psi_{N}(t^{k})\big).u\neq0$ for any $k\in\Z$.
Moreover, up to a permutation, we have
\begin{itemize}
\item[(i)] For $1\leq n\leq r_{-1}$, we have \[b_{n}=-l_{-1,n}(N+1)\ \text{and}\  v_{n}=\rD(N-1,k+k_{-1,n})\prod_{i=N-1}^{-1}\prod_{j=1}^{r_{i}}\rD(i,k_{i,j})^{l_{i,j}-\delta_{-1,i}\delta_{j,n}}v_{\psi},\]
\item[(ii)] For $r_{-1}+1\leq n\leq q$, we have $\mathrm{lth}(v_{n})<\mathrm{lth}(u)$,
\item[(iii)] If there exists some $v_{n}$ $(r_{-1}+1\leq n\leq q)$
such that $\mathrm{lth}_{-1}(v_{n})=\mathrm{lth}_{-1}(u)$, then
$\mathcal{T}_{\mathrm{set}}(v_{n})\subseteq\mathcal{T}_{\mathrm{set}}(u)$.
\end{itemize}
\end{proposition}

\begin{proof} (1) Note that $\mathrm{lth}_{-1}(u)=0$, we obtain (1) by Lemma \ref{direct-calculate}.

(2) By Lemma \ref{direct-calculate}, (i) and (ii) of $(2)$ follows from a direct computation.
For (iii) of $(2)$, it is sufficient to prove the following claim.

\vspace{2mm}

\noindent {\bf Claim}. For any $m\geq N$, $k\in\Z$,
and $u\in B(\fg_{\geq N},\psi)$ be as in \eqref{typical-vector}, write
\[\big(\rD(m,k)-\psi_{m}(t^{k})\big).u=\sum_{n=1}^{p}a_{n}w_{n},\]
where $p\in\N$, $a_{1},a_{2},\dots,a_{p}\in\C^{\times}$,
$w_{1},w_{2},\dots,w_{p}\in B(\fg_{\geq N},\psi)$ are distinct elements.
If $\mathrm{lth}_{-1}(w_{n})=\mathrm{lth}_{-1}(u)$ for some $1\leq n\leq p$,
then $\mathcal{T}_{\mathrm{set}}(w_{n})\subseteq\mathcal{T}_{\mathrm{set}}(u)$.

\vspace{2mm}

It is straightforward to
prove the claim by induction on $\mathrm{lth}(u)$.
\end{proof}

We also have the following property of Whittaker vectors in $W(\fg_{\geq N},\psi)$.
\begin{proposition}\label{some-property-Whit}
Let $v=\sum_{i=1}^{p}a_{i}v_{i}\in W(\fg_{\geq N},\psi)_{\psi}\setminus \C v_{\psi}$,
where $a_{1},a_{2},\dots,a_{p}\in\C^{\times}$ and $v_{1},v_{2},\dots,v_{p}\in B(\fg_{\geq N},\psi)$ are distinct elements.
Then \[\mathrm{lth}_{-1}(v_{i})=0\quad \text{for all}\quad i=1,\dots,p.\]
\end{proposition}
\begin{proof}
Suppose to the contrary, then we have \[I:=\{1\leq i\leq p\mid\mathrm{lth}_{-1}(v_{i})\geq1\}\neq\emptyset.\]
Without loss of generality, we assume that
\[I=\{1,2,\dots,i_{0}\}\]
for some $1\leq i_{0}\leq n$ and $v_{1}\succ\cdots\succ v_{i_{0}}$.

\vspace{2mm}

\noindent {\bf Claim}. $\big(\rD(N,k)-\psi_{N}(t^{k})\big).v\neq0$ for all sufficiently large integers $k$.

In fact, from Proposition \ref{ele-lemma}(1), we only need to note that
if we write \[\big(\rD(N,k)-\psi_{N}(t^{k})\big).\sum_{j=i_{0}+1}^{p}a_{j}v_{j}=
\sum_{m=1}^{q}b_{m}u_{m},\quad k\in\Z,\]
where $q\in\N$, $b_{1},\dots,b_{q}\in\C^{\times}$,
and $u_{1},\dots,u_{q}\in B(\fg_{\geq N},\psi)$ are distinct elements,
then we have
\[\mathcal{T}_{\mathrm{set}}(u_{m})\subseteq\bigcup_{j=i_{0}+1}^{p}\mathcal{T}_{\mathrm{set}}(v_{j})
\quad \text{for all}\quad m=1,\dots,q.\]
Let $v_{1}$ be as in \eqref{typical-vector} with $\mathrm{lth}_{-1}(v_{1})\geq1$.
Now for sufficiently large integer $k$, by Proposition \ref{ele-lemma}(2), we have
\[\big(\rD(N,k)-\psi_{N}(t^{k})\big).\sum_{j=1}^{i_{0}}a_{j}v_{j}=c_{0}w_{0}+\sum_{m=1}^{q'}c_{m}w_{m},\]
where $c_{0}:=-a_{1}l_{-1,1}(N+1),c_{1},\dots,c_{q'}\in\C^{\times}$ and
\[w_{0}:=\rD(N-1,k+k_{-1,1})\prod_{i=N-1}^{-1}\prod_{j=1}^{r_{i}}\rD(i,k_{i,j})^{l_{i,j}-\delta_{-1,i}\delta_{1,j}}v_{\psi},
w_{1},\dots,w_{p'}\in B(\fg_{\geq N},\psi)\] such that $w_{0}\succ w_{m}$ for $m=1,\dots,p'$.
Note that for sufficiently large integer $k$, we have
$k+k_{-1,r_{-1}}\notin \bigcup_{j=i_{0}+1}^{p}\mathcal{T}_{\mathrm{set}}(v_{j})$.
This proves the claim, which contradicts the assumption that $v\in W(\fg_{\geq N},\psi)_{\psi}$. We complete the lemma.
\end{proof}

In the rest of this section, we assume that either $\psi_{2N-1}\notin\mathcal{E}$ or $\psi_{2N}\notin\mathcal{E}$.

\begin{proposition}\label{EXP-2N}
If $\psi_{2N}\notin\mathcal{E}$, then $W(\fg_{\geq N},\psi)_{\psi}=\C v_{\psi}$.
\end{proposition}
\begin{proof}
Suppose to the contrary that $W(\fg_{\geq N},\psi)_{\psi}\neq\C v_{\psi}$.
Then from Proposition \ref{some-property-Whit}, we know that
there exists some \[v=\sum_{i=1}^{p}a_{i}v_{i}\in W(\fg_{\geq N},\psi)_{\psi}\setminus \C v_{\psi},\]
where $a_{1},a_{2},\dots,a_{p}\in\C^{\times}$ and
$v_{1},v_{2},\dots,v_{p}\in B(\fg_{\geq N},\psi)\setminus\{v_{\psi}\}$ are distinct elements
such that $\mathrm{lth}_{-1}(v_{i})=0$ for $i=1,\dots,p$.

Set \[\ell=\min\,\bigcup_{i=1}^{p}\mathcal{D}_{\mathrm{set}}(v_{i}).\]
Recall $\psi_{i}=0$ for all $i\geq2N+1$ from \eqref{psi-n}. By Lemma \ref{direct-calculate}, we know that
\[\big(\rD(2N-\ell,k)-\psi_{2N-\ell}(t^{k})\big).\omega=0\]
for all $k\in\Z$ and $\omega\in B(\fg_{\geq N},\psi)$ such that $\min\mathcal{D}_{\mathrm{set}}(\omega)>\ell$.

Without loss of generality, we assume that $\min\mathcal{D}_{\mathrm{set}}(v_{i})=\ell$
for $1\leq i\leq p_{0}$ and $\min\,\mathcal{D}_{\mathrm{set}}(v_{i})>\ell$ for $p_{0}+1\leq i\leq p$,
where $1\leq p_{0}\leq p$. We may further assume that $v_{1}\succ\cdots\succ v_{p_{0}}$.
For convenience, we write \[v_{1}=x_{N-1}\cdots x_{\ell+1}\rD(\ell,k_{1})^{l_{1}}\cdots\rD(\ell,k_{s})^{l_{s}}v_{\psi},\]
where $x_{N-1}\in\U(\fg_{N-1}),\dots,x_{\ell+1}\in\U(\fg_{\ell+1})$ are nonzero elements, $k_{1}>\cdots>k_{s}$,
and $l_{1},\dots,l_{s}\in\Z_{+}$. It is obvious that there exists $1\leq q\leq p_{0}$ such that
\[v_{i}=x_{N-1}\cdots x_{\ell+1}\rD(\ell,k_{1})^{l_{1}}\cdots\rD(\ell,k_{s})^{l_{s}-1}\rD(\ell,k_{s_{i}})v_{\psi},\]
where $i=2,\dots,q$, $k_{s}>k_{s_{2}}>k_{s_{3}}>\cdots>k_{s_{q}}$, and $v_{q+1}\prec \xi$
for any $\xi\in\{x_{N-1}\cdots x_{\ell+1}\rD(\ell,k_{1})^{l_{1}}\cdots
\rD(\ell,k_{s})^{l_{s}-1}\rD(\ell,k)v_{\psi}\mid k\in\Z,\ k<k_{s}\}$.
\vspace{3mm}

For $k\in\Z$, we have
$\big(\rD(2N-\ell,k)-\psi_{2N-\ell}(t^{k})\big).v_{1}=b_{0}\omega_{0}+\sum_{m=1}^{p'}b_{m}\omega_{m}$,
where $b_{0}:=l_{s}(2\ell-2N)\psi_{2N}(t^{k+k_{s}}),\ b_{1},\dots,b_{m}\in\C$ and
\[\omega_{0}:=x_{N-1}\cdots x_{\ell+1}\rD(\ell,k_{1})^{l_{1}}\cdots
\rD(\ell,k_{s})^{l_{s}-1}v_{\psi},\ \omega_{1},\dots,\omega_{m}\in B(\fg_{\geq N},\psi)\]
such that $\omega_{0}\succ\omega_{m}$ for $m=1,\dots,p'$. Similarly, we can deduce that
\begin{align}\label{Imp-1}\big(\rD(2N-\ell,k)-\psi_{2N-\ell}(t^{k})\big).v
=c_{0}\omega_{0}+\sum_{m=1}^{q'}c_{m}\omega_{m}',\end{align}
where $c_{0}:=a_{1}l_{s}(2\ell-2N)\psi_{2N}(t^{k+k_{s}})+
\sum_{n=2}^{q}a_{n}(2\ell-2N)\psi_{2N}(t^{k+k_{s_{n}}}),\ c_{1},\dots,c_{q'}\in\C$
and $\omega_{0},\omega_{1}',\dots,\omega_{q'}'\in B(\fg_{\geq N},\psi)$ such that $\omega_{0}\succ\omega_{m}'$
for $m=1,\dots,q'$. This forces $c_{0}=0$ for all $k\in\Z$, which means that
\[\Big(\big(a_{1}l_{s}(2\ell-2N)t^{k_{s}}+\sum_{n=2}^{q}a_{n}(2\ell-2N)t^{k_{s_{n}}}\big).\psi_{2N}\Big)(t^{k})=0\]
for all $k\in\Z$. That is $\big(a_{1}l_{s}(2\ell-2N)t^{k_{s}}+\sum_{n=2}^{q}a_{n}(2\ell-2N)t^{k_{s_{n}}}\big).\psi_{2N}=0$.
Since $a_{1}l_{s}(2\ell-2N)t^{k_{s}}+\sum_{n=2}^{q}a_{n}(2\ell-2N)t^{k_{s_{n}}}\in\C[t,t^{-1}]$ is non-zero,
it follows from Lemma \ref{lem:charCE} that $\psi_{2N}\in\mathcal{E}$, which is a contradiction.
This completes the proof.
\end{proof}

Now we remain to prove the case that $\psi_{2N}\in\mathcal{E}$ and $\psi_{2N-1}\notin\mathcal{E}$.
\begin{proposition}\label{EXP-2N-1}
If $\psi_{2N}\in\mathcal{E}$ and $\psi_{2N-1}\notin\mathcal{E}$, then $W(\fg_{\geq N},\psi)_{\psi}=\C v_{\psi}$.
\end{proposition}
\begin{proof}
Suppose to the contrary that $W(\fg_{\geq N},\psi)_{\psi}\neq\C v_{\psi}$.
Then from Proposition \ref{some-property-Whit}, we know that
there exists some \[v=\sum_{i=1}^{P}\alpha_{i}\nu_{i}\in W(\fg_{\geq N},\psi)_{\psi}\setminus \C v_{\psi},\]
where $\alpha_{1},\dots,\alpha_{P}\in\C^{\times}$ and $\nu_{1},\dots,\nu_{P}\in B(\fg_{\geq N},\psi)\setminus\{v_{\psi}\}$
are distinct elements such that $\mathrm{lth}_{-1}(\nu_{i})=0$ for $i=1,\dots,P$.

Set \[\ell=\min\,\bigcup_{i=1}^{P}\mathcal{D}_{\mathrm{set}}(\nu_{i})\quad \text{and}\quad
\iota=\max\,\{\mathrm{lth}_{\ell}(\nu_{i})\mid i=1,\dots,P\}.\]
For convenience, we rewrite $v$ as
\[v=\sum_{i=1}^{p}a_{i}v_{i}+\sum_{j=1}^{Q}\beta_{j}\mu_{j}+\sum_{m=1}^{Q'}\gamma_{m}u_{m},\]
where the parameters satisfy:
\begin{itemize}
\item $p\in\Z_{+}$, $Q,Q'\in\N$, $a_{1},\dots,a_{p},\beta_{1},\dots,\beta_{Q},\gamma_{1},\dots,\gamma_{Q'}\in\C^{\times}$;
\item $v_{1},\dots,v_{p}\in B(\fg_{\geq N},\psi)$ are distinct elements such that $v_{1}\succ v_{2}\succ\cdots\succ v_{p}$
and $\mathrm{lth}_{\ell}(v_{i})=\iota$ for $i=1,\dots,p$;
\item $\mu_{1},\dots,\mu_{Q}\in B(\fg_{\geq N},\psi)$ are distinct elements
such that $\mathrm{lth}_{\ell}(\mu_{j})=\iota-1$ for $j=1,\dots,Q$;
\item $u_{1},\dots,u_{Q'}\in B(\fg_{\geq N},\psi)$ are distinct elements
such that $\mathrm{lth}_{\ell}(u_{m})\leq\iota-2$ for $m=1,\dots,Q'$.
\end{itemize}
Set $v_{1}=x_{N-1}\cdots x_{\ell+2}\rD(\ell+1,m_{1})^{n_{1}}\cdots\rD(\ell+1,m_{r})^{n_{r}}
\rD(\ell,k_{1})^{l_{1}}\cdots\rD(\ell,k_{s})^{l_{s}}v_{\psi}$,
where $x_{N-1}\in\U(\fg_{N-1}),\dots,x_{\ell+2}\in\U(\fg_{\ell+2})$ are nonzero elements, $r\in\N$,
$m_{r}<\cdots<m_{1}$, $n_{1},\dots,n_{r}\in\Z_{+}$, $k_{s}<\cdots<k_{1}$,
and $l_{1},\dots,l_{s}\in\Z_{+}$ with $l_{1}+\cdots+l_{s}=\iota$.
Write \[x_{\ell+1}:=\rD(\ell+1,m_{1})^{n_{1}}\cdots\rD(\ell+1,m_{r})^{n_{r}}.\]

Without loss of generality, we assume $\sum_{j=1}^{Q}\beta_{j}\mu_{j}\neq0$.
We rewrite $\sum_{j=1}^{Q}\beta_{j}\mu_{j}$ as
\[\sum_{j=1}^{Q}\beta_{j}\mu_{j}=\sum_{n\in\Z}b_{n}x_{N-1}\cdots x_{\ell+1}\rD(\ell+1,n)\rD(\ell,k_{1})^{l_{1}}\cdots\rD(\ell,k_{s})^{l_{s}-1}v_{\psi}+\sum_{j=1}^{q}\lambda_{j}w_{j},\]
where the parameters satisfy:
\begin{itemize}
\item[(c1)] there are only finitely many $b_{n}\in\C$ ($n\in\Z$) are non-zero;
\item[(c2)] $q\in\N$, $w_{1},\dots,w_{q}\in B(\fg_{\geq N},\psi)$ are distinct elements such that
$\mathrm{lth}_{\ell}(w_{j})=\iota-1$ for $j=1,\dots,q$;
\item[(c3)] for each $j=1,\dots,q$, either $\mathcal{T}_{\ell}(w_{j})\neq \mathcal{T}_{\ell}(v_{+})$ or
$\mathcal{T}_{\ell}(w_{j})=\mathcal{T}_{\ell}(v_{+})$ but $w_{j}\not\in\{x_{N-1}\cdots x_{\ell+1}\rD(\ell+1,n)\rD(\ell,k_{1})^{l_{1}}\cdots\rD(\ell,k_{s})^{l_{s}-1}v_{\psi}\mid n\in\Z\}$,
where \[v_{+}=x_{N-1}\cdots x_{\ell+2}\rD(\ell+1,m_{1})^{n_{1}}\cdots\rD(\ell+1,m_{r})^{n_{r}}
\rD(\ell,k_{1})^{l_{1}}\cdots\rD(\ell,k_{s})^{l_{s}-1}v_{\psi}.\]
\end{itemize}

\vspace{2mm}

For $k\in\Z$, we have
$$\big(\rD(2N-1-\ell,k)-\psi_{2N-1-\ell}(t^{k})\big).v_{1}=l_{s}(2\ell-2N+1)
\psi_{2N-1}(t^{k+k_{s}})v_{+}+\sum_{i=1}^{p''}a_{i}''v_{i}'',$$
where $a_{1}'',\dots,a_{p''}''\in\C^{\times}$ and $v_{+},v_{1}'',\dots,v_{p''}''\in B(\fg_{\geq N},\psi)$
are distinct elements such that $v_{i}''\prec v_{+}$ for $i=1,\dots,p''$.

\vspace{2mm}
Similar to the proof of Proposition \ref{EXP-2N}, we know that there exists $1\leq p_{0}\leq p$ such that
\[v_{i}=x_{N-1}\cdots x_{\ell+2}\rD(\ell+1,m_{1})^{n_{1}}\cdots\rD(\ell+1,m_{r})^{n_{r}}
\rD(\ell,k_{1})^{l_{1}}\cdots\rD(\ell,k_{s})^{l_{s}-1}\rD(\ell,k_{s_{i}})v_{\psi}\]
for $i=2,\dots,p_{0}$, where $k_{s}>k_{s_{2}}>\cdots>k_{s_{p_{0}}}$ and that
\[v_{p_{0}+1}\prec x_{N-1}\cdots x_{\ell+2}\rD(\ell+1,m_{1})^{n_{1}}\cdots\rD(\ell+1,m_{r})^{n_{r}}
\rD(\ell,k_{1})^{l_{1}}\cdots\rD(\ell,k_{s})^{l_{s}-1}\rD(\ell,m)v_{\psi}\]
for any $m\in\Z$ with $m<k_{s}$. Then we obtain that, for $k\in\Z$,
$$\big(\rD(2N-1-\ell,k)-\psi_{2N-1-\ell}(t^{k})\big).\sum_{i=1}^{p}a_{i}v_{i}=b_{0}v_{+}+\sum_{i=1}^{p'}b_{i}'v_{i}',$$
where $b_{0}:=l_{s}a_{1}(2\ell-2N+1)\psi_{2N-1}(t^{k+k_{s}})+\sum_{i=2}^{p_{0}}a_{i}(2\ell-2N+1)\psi_{2N-1}(t^{k+k_{s_{i}}}),\
b_{1}',\dots,b_{p'}'\in\C$ and $v_{1}',\dots,v_{p'}'\in B(\fg_{\geq N},\psi)$ are distinct elements such that
$v_{i}'\prec v_{+}$ for $i=1,\dots,p'$.

\vspace{2mm}

\noindent {\bf Claim 1}.\quad For $k\in\Z$, we have
$$\big(\rD(2N-1-\ell,k)-\psi_{2N-1-\ell}(t^{k})\big).\sum_{j=1}^{Q}\beta_{j}\mu_{j}=c_{0}v_{+}+\sum_{j=1}^{q'}c_{j}'w_{j}',$$
where $c_{0}:=\sum_{n\in\Z}b_{n}(2\ell+2-2N)\big(1+\sum_{i=1}^{r}n_{i}\delta_{n,m_{i}}\big)\psi_{2N}(t^{k+n}),\ c_{1}',\dots,
c_{q'}'\in\C$
and $w_{1}',\dots,w_{q'}'\in B(\fg_{\geq N},\psi)\setminus\{v_{+}\}$.

Note that $\ell+1\geq1$, which implies that we only need to collect the elements
appearing in $\sum_{j=1}^{Q}\beta_{j}\mu_{j}$ that have the form
\[x_{N-1}\cdots x_{\ell+1}\rD(\ell+1,n)\rD(\ell,k_{1})^{l_{1}}\cdots\rD(\ell,k_{s})^{l_{s}-1}v_{\psi},\quad n\in\Z.\]
%We can deduce that $\rD(\ell+1,n)$ ($n\in\Z$) occurs only once.
Then from (c1)-(c3), we obtain Claim 1.

\noindent {\bf Claim 2}.\quad For $k\in\Z$, we have
$$\big(\rD(2N-1-\ell,k)-\psi_{2N-1-\ell}(t^{k})\big).\sum_{m=1}^{Q'}\gamma_{m}u_{m}=\sum_{j=1}^{Q''}\gamma_{j}'z_{j},$$
where $\gamma_{1}',\dots,\gamma_{Q''}'\in\C$ and $z_{1},\dots,z_{Q''}\in B(\fg_{\geq N},\psi)\setminus\{v_{+}\}$.

In fact, we have $\mathrm{lth}_{\ell}(z_{j})\leq\iota-2$ for all $j=1,\dots,Q''$.
Thus $z_{j}\neq v_{+}$ for all $j=1,\dots,Q''$.

\vspace{3mm}

Since $\big(\rD(2N-1-\ell,k)-\psi_{2N-1-\ell}(t^{k})\big).v=0$ for all $k\in\Z$, this forces
$b_{0}+c_{0}=0$ for all $k\in\Z$.
That is,
$$\begin{aligned}
\Big[\Big(\big(l_{s}a_{1}(2\ell-2N+1)t^{k_{s}}+\sum_{i=2}^{p_{0}}a_{i}(2\ell-2N+1)t^{k_{s_{i}}}\big).\psi_{2N-1}\Big)
+\\ \quad \Big(\big(\sum_{n\in\Z}b_{n}(2\ell+2-2N)\big(1+\sum_{i=1}^{r}n_{i}\delta_{n,m_{i}}\big)t^{n}\big).\psi_{2N}\Big)\Big](t^{k})=0
\end{aligned}$$
for all $k\in\Z$, which means
$$\begin{aligned}
\Big(l_{s}a_{1}(2\ell-2N+1)t^{k_{s}}+\sum_{i=2}^{p_{0}}a_{i}(2\ell-2N+1)t^{k_{s_{i}}}\Big).\psi_{2N-1}
+\\ \quad \Big(\sum_{n\in\Z}b_{n}(2\ell+2-2N)\big(1+\sum_{i=1}^{r}n_{i}\delta_{n,m_{i}}\big)t^{n}\Big).\psi_{2N}=0
\end{aligned}$$

Note that \[l_{s}a_{1}(2\ell-2N+1)t^{k_{s}}+\sum_{i=2}^{p_{0}}a_{i}(2\ell-2N+1)t^{k_{s_{i}}}\in\mathcal{A}\]
is non-zero element. Since $\psi_{2N}\in\mathcal{E}$ and $\psi_{2N-1}\notin\mathcal{E}$,
from Lemmas \ref{lem:charCE} and \ref{lem:non-CE}, we get a contradiction.
Note that even if all $b_{n}=0$ for $n\in\Z$, we also can get the contradiction. This completes the proof.
\end{proof}

%\begin{remark}
%{\color{red}In the proof of Proposition \ref{EXP-2N-1}, using $\psi_{2N-1}\notin\mathcal{E}$,
%similarly to that of Proposition \ref{EXP-2N}, it's possible that one would consider
%\[\big(\rD(2N-1-\ell-1,k)-\psi_{2N-1-\ell-1}(t^{k})\big).v \quad k\in\Z.\]
%But this will be ineffective when $x_{\ell+1}=1$. And our method still holds even if $N=1$ and $\ell=0$.}
%\end{remark}

{\em Proof of Theorem \ref{first-main-restlt}}. From Propositions
\ref{psi-reducible}, \ref{EXP-2N}, \ref{EXP-2N-1} and Lemmas \ref{ALZ}, \ref{CJ},
Theorem \ref{first-main-restlt} follows.\hfill$\Box$

For any $f(t)=\sum_{i=-r}^{r}a_{i}t^{i}\in\mathcal{A}$, $a_{i}\in\C$, we write
\[\rD(n,f)=\sum_{i=-r}^{r}a_{i}\rD(n,i),\ n\geq-1.\]

For the general theory of Whittaker modules, it is difficult to determine all Whittaker vectors.
We present certain Whittaker vectors in the following result.
\begin{proposition}\label{Whittaker-vector-N}
Let $\psi_{2N-1},\psi_{2N}\in\mathcal{E}$. For any
$f_{1}(t),\dots, f_{s}(t)\in\mathrm{Ann}(\psi_{2N-1})\cap\mathrm{Ann}(\psi_{2N})$, we have
\[\prod_{j=1}^{s}\rD(N-1,f_{j})v_{\psi}\in W(\fg_{\geq N},\psi)_{\psi}.\]
\end{proposition}
\begin{proof}
For $n\geq N$, we note that
\begin{eqnarray*}&&\big(\rD(n,k)-\psi_{n}(t^{k})\big).\prod_{j=1}^{s}\rD(N-1,f_{j})v_{\psi}=(N-1-n)
(\rd_{n+N-1}\otimes t^{k}f_{1}).\prod_{j=2}^{s}\rD(N-1,f_{j})v_{\psi}\\&&
\qquad\qquad\qquad+\rD(N-1,f_{1})\Big(\big(\rD(n,k)-\psi_{n}(t^{k})\big).\prod_{j=2}^{s}\rD(N-1,f_{j})v_{\psi}\Big).\end{eqnarray*}
Since $f_{1}\in\mathrm{Ann}(\psi_{2N-1})\cap\mathrm{Ann}(\psi_{2N})$ and $\psi_{i}=0$ for $i\geq2N+1$,
we have $\psi_{n+N-1}(t^{k}f_{1})=(f_{1}(t).\psi)(t^{k})=0$ for $n\geq N$. Then
\[(\rd_{n+N-1}\otimes t^{k}f_{1}(t)).\prod_{j=2}^{s}\rD(N-1,f_{j})v_{\psi}=\big(\rd_{n+N-1}\otimes t^{k}f_{1}-\psi_{n+N-1}(t^{k}f_{1})\big).\prod_{j=2}^{s}\rD(N-1,f_{j})v_{\psi}.\]
By induction on $s$, we obtain $(\rd_{n+N-1}\otimes t^{k}f_{1}(t)).\prod_{j=2}^{s}\rD(N-1,f_{j})v_{\psi}=0$
and $\big(\rD(n,k)-\psi_{n}(t^{k})\big).\prod_{j=2}^{s}\rD(N-1,f_{j})v_{\psi}=0$. This gives
\[\big(\rD(n,k)-\psi_{n}(t^{k})\big).\prod_{j=1}^{s}\rD(N-1,f_{j})v_{\psi}=0\] for $n\geq N$, i.e.,
$\prod_{j=1}^{s}\rD(N-1,f_{j})v_{\psi}\in W(\fg_{\geq N},\psi)_{\psi}$.
\end{proof}

\section{Whittaker modules associated to $(\fg,\fg_{-})$}
Let $\varphi:\fg_{-}\rightarrow\C$ be a Whittaker function, i.e., $\varphi\in\fg_{-}^{*}$.
$\varphi$ induces a linear functional on $\mathcal{A}$, and we denote it also by $\varphi$. Namely,
\[\varphi(t^{n}):=\varphi(\rd_{-1}\otimes t^{n}),\quad n\in\Z.\]

In this section, we study the simplicity of Whittaker modules $W(\fg_{-},\varphi)$.
Now we state the main result of this section, and we prove it
by using Proposition \ref{E-Reducible} to Proposition \ref{var-not-exp}.
\begin{thm}\label{second-main-restlt}Let $\varphi:\fg_{-}\rightarrow\C$ be a Whittaker function.
Then the universal Whittaker $\fg$-module $W(\fg_{-},\varphi)$ is simple if and only if $\varphi\notin\mathcal{E}$.
\end{thm}

\begin{proposition}\label{E-Reducible}
If $\varphi\in\mathcal{E}$, then $W(\fg_{-},\varphi)$ is reducible.
\end{proposition}
\begin{proof}
Since $\varphi\in\mathcal{E}$, there exists a non-zero polynomial
$c(t)=\sum_{j=0}^{q}c_{j}t^{j}$ with degree $q\geq1$ such that
$c(t)\in\mathrm{Ann}(\varphi)$. Similar to the proof of Proposition \ref{psi-reducible},
we can prove that the submodule $\U(\fg)u$ is a proper submodule of $W(\fg_{-},\varphi)$,
where $u=\sum_{j=0}^{q}c_{j}\rD(0,j)v_{\varphi}$.
\end{proof}

Set\begin{eqnarray*}&&B(\fg_{-},\varphi)=\{\rD(r_{1},k_{r_{1},1})\rD(r_{1},k_{r_{1},2})\cdots
\rD(r_{1},k_{r_{1},s_{1}})\rD(r_{2},k_{r_{2},1})\cdots
\rD(r_{2},k_{r_{2},s_{2}})\\&&
\quad\cdots\cdots\rD(r_{n},k_{r_{n},1})\rD(r_{n},k_{r_{n},2})\cdots
\rD(r_{n},k_{r_{n},s_{n}})v_{\varphi}\mid n\in\N,\ k_{r_{i},j}\in\Z\\&&\quad
0\leq r_{n}<\cdots<r_{1},\ k_{r_{i},s_{i}}\leq\cdots\leq k_{r_{i},2}\leq k_{r_{i},1},\
1\leq i\leq n,\ 1\leq j\leq s_{i}\},\end{eqnarray*}
%\begin{eqnarray*}B^{-}(\psi)&=&\{\rD(r_{1},k_{1,r_{1}})^{l_{1,r_{1}}}\rD(r_{1},k_{2,r_{1}})^{l_{2,r_{1}}}\cdots
%\rD(r_{1},k_{s_{r_{1}},r_{1}})^{l_{s_{r_{1}},r_{1}}}\\&&
%\rD(r_{2},k_{1,r_{2}})^{l_{1,r_{2}}}\rD(r_{2},k_{2,r_{2}})^{l_{2,r_{2}}}\cdots
%\rD(r_{2},k_{s_{r_{2}},r_{2}})^{l_{s_{r_{2}},r_{2}}}\\&&
%\cdots\cdots\cdots\cdots\\&&
%\rD(r_{n},k_{1,r_{n}})^{l_{1,r_{n}}}\rD(r_{n},k_{2,r_{n}})^{l_{2,r_{n}}}\cdots
%\rD(r_{n},k_{s_{r_{n}},r_{n}})^{l_{s_{r_{n}},r_{n}}}v_{\psi}\mid\\&&
%n\in\N,\ r_{1}>r_{2}>\cdots>r_{n}\geq0,\ l_{i,r_{j}}\in\Z_{+},\ k_{i,r_{j}}\in\Z\\&&
%k_{1,r_{j}}>k_{2,r_{j}}>\cdots>k_{s_{j},r_{s_{j}}},\ j=1,2,\dots,n,\ i=1,2,\dots,r_{j}\},\end{eqnarray*}
which forms a basis of $W(\fg_{-},\varphi)$.
For a vector \begin{eqnarray}\label{Topical-vector--}u=\prod_{i=1}^{n}\prod_{j=1}^{s_{i}}
\rD(r_{i},k_{r_{i},j})^{l_{r_{i},j}}v_{\varphi}\in B(\fg_{-},\varphi),\end{eqnarray}
where $0\leq r_{n}<\cdots<r_{1}$, $l_{r_{i},j}\in\Z_{+}$, and $k_{r_{i},s_{i}}<\cdots<k_{r_{i},2}<k_{r_{i},1}$,
we let \[\mathrm{lth}_{r_{i}}(u)=\sum_{j=1}^{s_{i}}l_{r_{i},j},\quad
\mathrm{lth}(u)=\sum_{i=1}^{n}\sum_{j=1}^{s_{i}}l_{r_{i},j},\]
\[\mathcal{D}(u)=\big(\underbrace{r_{1},r_{1},\dots,r_{1}}_{\mathrm{lth}_{r_{1}}(u)-\mathrm{times}},
\underbrace{r_{2},r_{2},\dots,r_{2}}_{\mathrm{lth}_{r_{2}}(u)-\mathrm{times}},\dots,
\underbrace{r_{n},r_{n},\dots,r_{n}}_{\mathrm{lth}_{r_{n}}(u)-\mathrm{times}}\big),\]
$$\begin{aligned}\mathcal{T}(u)=\big(\underbrace{k_{r_{1},1},k_{r_{1},1},\dots,k_{r_{1},1}}_{l_{r_{1},1}-\mathrm{times}},
\underbrace{k_{r_{1},2},k_{r_{1},2},\dots,k_{r_{1},2}}_{l_{r_{1},2}-\mathrm{times}},\dots,
\underbrace{k_{r_{1},s_{1}},k_{r_{1},s_{1}},\dots,k_{r_{1},s_{1}}}_{l_{r_{1},s_{1}}-\mathrm{times}},\\ \cdots,
\underbrace{k_{r_{n},1},k_{r_{n},1},\dots,k_{r_{n},1}}_{l_{r_{n},1}-\mathrm{times}},\cdots,
\underbrace{k_{r_{n},s_{n}},k_{r_{n},s_{n}},\dots,k_{r_{n},s_{n}}}_{l_{r_{n},s_{n}}-\mathrm{times}}\big),\end{aligned}$$
\[\mathcal{D}_{\mathrm{set}}(u)=\{r_{1},\dots,r_{n}\},\quad\mathcal{T}_{r_{i},
\mathrm{set}}(u)=\{k_{r_{i},j}\mid j=1,2,\dots,s_{i}\}\ \text{for}\ i=1,2,\dots,n,\]and
$\mathcal{T}_{\mathrm{set}}(u)=\{k_{r_{i},j}\mid i=1,2,\dots,n,\ j=1,2,\dots,s_{i}\}$.

For any $r\in\Z_{+}$, recall that the total order on $\Z^{r}$ in \eqref{total-order}.
We define the principle total order ``$\succ$'' on $B(\fg_{-},\varphi)$ as follows:
for different $u,v\in B(\fg_{-},\varphi)$, set $u\succ v$ if and only if one of the following conditions satisfy:
     \begin{itemize}
     \item $\text{lth}(u)>\text{lth}(v)$;
     \item $\text{lth}(u)=\text{lth}(v)$ and $\mathcal{D}(u)\succ\mathcal{D}(v)$
     under the order $\succ$ on $\Z^{\text{lth}(u)}$;
     \item $\text{lth}(u)=\text{lth}(v)$, $\mathcal{D}(u)=\mathcal{D}(v)$,
     and $\mathcal{T}(u)\succ\mathcal{T}(v)$
     under the order $\succ$ on $\Z^{\text{lth}(u)}$.
     \end{itemize}

We need the following lemmas.
\begin{lemma}\label{>0-case}
Let $u\in B(\fg_{-},\varphi)$ be as in \eqref{Topical-vector--}.
Suppose that $\mathrm{lth}(u)\geq1$ and $r_{n}>0$. Then we have\\
(1) $\big(\rD(-1,k)-\varphi(t^{k})\big).u\neq0$ for all $k\in\Z$.\\
(2) for any $k\in\Z$, if we write $\big(\rD(-1,k)-\varphi(t^{k})\big).u=\sum_{s=1}^{p}a_{s}v_{s}$,
where $a_{1}.\dots,a_{p}\in\C^{\times}$ and $v_{1},\dots,v_{p}\in B(\fg_{-},\varphi)$
with $v_{1}\succ\cdots\succ v_{p}$, then
\[a_{1}=(r_{n}+1)l_{r_{n},s_{n}},\quad v_{1}=\prod_{i=1}^{n}\prod_{j=1}^{s_{i}}
\rD(r_{i},k_{r_{i},j})^{l_{r_{i},j}-\delta_{i,n}\delta_{j,s_{n}}}\rD(r_{n}-1,k+k_{r_{n},s_{n}})v_{\varphi}.\]
(3) there exists $N_{u}\in\Z$ such that if $k>N_{u}$, then
\[\mathcal{T}_{\mathrm{set}}(v_{s})\nsubseteq\mathcal{T}_{\mathrm{set}}(u)
\quad\text{for all}\ s=1,\dots,p.\]
\end{lemma}
\begin{proof}
$(2)$ is clear, and $(1)$ follows from $(2)$. For $(3)$, we only need to
notice that, for any $v_{s}$ ($1\leq s\leq p$), there exist some $0\leq a_{r_{i},j}\leq l_{r_{i},j}$,
$i=1,\dots,n$, $j=1,\dots,s_{i}$, such that
$k+\sum_{i=1}^{n}\sum_{j=1}^{s_{i}}a_{r_{i},j}k_{r_{i},j}\in \mathcal{T}_{\mathrm{set}}(v_{s})$.
\end{proof}

\begin{lemma}\label{=0-case}
Let $u\in B(\fg_{-},\varphi)$ be as in \eqref{Topical-vector--}.
Suppose $\mathrm{lth}(u)\geq1$ and $r_{n}=0$. Then we have\\
(1) $\big(\rD(-1,k)-\psi(t^{k})\big).u=u_{1}+u_{2}$, $k\in\Z$,
where \[u_{1}=\Big[\rD(-1,k),\prod_{i=1}^{n-1}\prod_{j=1}^{s_{i}}
\rD(r_{i},k_{r_{i},j})^{l_{r_{i},j}}\Big]\prod_{j=1}^{s_{n}}
\rD(r_{n},k_{r_{n},j})^{l_{r_{n},j}}v_{\varphi}\]
and
$u_{2}=\prod_{i=1}^{n-1}\prod_{j=1}^{s_{i}}
\rD(r_{i},k_{r_{i},j})^{l_{r_{i},j}}\Big[\rD(-1,k),\prod_{j=1}^{s_{n}}
\rD(r_{n},k_{r_{n},j})^{l_{r_{n},j}}\Big]v_{\varphi}$.\\
(2) Suppose $n\geq2$. Write $u_{1}=\sum_{r=1}^{p}a_{r}v_{r}$, $u_{2}=\sum_{m=1}^{q}b_{m}w_{m}$,
where $p,q\in\N$, $a_{1},\dots,a_{p},b_{1},\dots,b_{q}\in\C^{\times}$,
$v_{1},\dots,v_{p}\in B(\fg_{-},\varphi)$ are distinct elements such that $v_{1}\succ\cdots\succ v_{p}$,
and $w_{1},\dots,w_{q}\in B(\fg_{-},\varphi)$ are distinct elements
such that $w_{1}\succ\dots\succ w_{q}$. Then we have
\begin{enumerate}
\item[(i)] $u_{1}\neq0$, $a_{1}=l_{r_{n-1},s_{n-1}}(r_{n-1}+1)$, and
\[v_{1}=\prod_{i=1}^{n-1}\prod_{j=1}^{s_{i}}
\rD(r_{i},k_{r_{i},j})^{l_{r_{i},j}-\delta_{i,n-1}\delta_{j,s_{n-1}}}
\cdot \rD(r_{n-1}-1,k+k_{r_{n-1},s_{n-1}})\prod_{j=1}^{s_{n}}
\rD(r_{n},k_{r_{n},j})^{l_{r_{n},j}}v_{\varphi}.\]
\item[(ii)] $\mathrm{lth}_{0}(w_{m})<\mathrm{lth}_{0}(u)\leq\mathrm{lth}_{0}(v_{r})$ for
$m=1,\dots,q$ and $r=1,\dots,p$.
\item[(iii)] For $m=1,\dots,q$ and $k\in\Z$,
\[v_{1}\succ w_{m},\ \mathrm{lth}(w_{m})\leq \mathrm{lth}(u)-1,\quad\text{and}\quad
\mathcal{T}_{\mathrm{set}}(w_{m})\subseteq\mathcal{T}_{\mathrm{set}}(u).\]
\item[(iv)] There exists $N_{u}\in\Z$ such that if $k>N_{u}$, then
$\mathcal{T}_{\mathrm{set}}(v_{r})\nsubseteq\mathcal{T}_{\mathrm{set}}(u)$ for all $r=1,\dots,p$.
\end{enumerate}

\end{lemma}
\begin{proof}
$(1)$ and (i) of $(2)$ are clear. For (ii) and (iii) of $(2)$,
we only need to note that $r_{n}=0$. For (iv) of $(2)$, we can
prove it similar to that of Lemma \ref{>0-case}(3).
\end{proof}

Now we characterize the Whittaker vector in $W(\fg_{-},\varphi)$, the following result is essential.

\begin{proposition}\label{reduce-to-0}
If there exists a vector
\[u=\sum_{m=1}^{p}a_{m}v_{m}\in W(\fg_{-},\varphi)_{\varphi}\setminus\C v_{\varphi},\]
where $a_{1},\dots,a_{p}\in\C^{\times}$ and $v_{1},\dots,v_{p}\in B(\fg_{-},\varphi)$ are distinct elements,
then we have
\[\text{either}\ v_{m}=v_{\varphi}\ \text{or}\ \mathcal{D}_{\mathrm{set}}(v_{i})=\{0\}\]
for $m=1,\dots,p$.
\end{proposition}
\begin{proof}Without loss of generality, we assume $v_{m}\neq v_{\varphi}$ for all $m=1,\dots,p$.
Since $u\in W(\fg_{-},\varphi)_{\varphi}$, we know that
$\big(\rD(-1,k)-\varphi(t^{k})\big).u=0$ for all $k\in\Z$.

\noindent {\bf Claim 1}
\qquad $\{0\}\subset\mathcal{D}_{\mathrm{set}}(v_{m})\quad\text{for all}\ m=1,\dots,p$.

\noindent Suppose to the contrary that there exists some $v_{m_{0}}$ such that
$\{0\}\not\subset\mathcal{D}_{\mathrm{set}}(v_{m_{0}})$, where $1\leq m_{0}\leq p$. We assume that
\[\{0\}\subset\mathcal{D}_{\mathrm{set}}(v_{m})\ \text{for all}\ 1\leq m< m_{0}\ \text{and}\
\{0\}\not\subset\mathcal{D}_{\mathrm{set}}(v_{j})\ \text{for all}\ m_{0}\leq j\leq p.\]
We may further assume that $v_{m_{0}}\succ\cdots\succ v_{p}$.
Set \[v_{m_{0}}=\prod_{i=1}^{n}\prod_{j=1}^{s_{i}}
\rD(r_{i},k_{r_{i},j})^{l_{r_{i},j}}v_{\varphi}\in B(\fg_{-},\varphi).\]
Then $1\leq r_{n}<\cdots<r_{2}<r_{1}$. It follows from Lemma \ref{>0-case} that we have
\begin{eqnarray}\label{Whit-varphi-1}\big(\rD(-1,k)-\varphi(t^{k})\big).\sum_{m=m_{0}}^{p}a_{m}v_{m}
=(r_{n}+1)l_{r_{n},s_{n}}\omega_{1}+\sum_{j=1}^{p'}b_{j}'w_{j},
\end{eqnarray}
where $b_{1}',\dots,b_{p'}'\in\C^{\times}$, $w_{1},\dots,w_{p'}\in B(\fg_{-},\varphi)$
with $\omega_{1}\succ w_{j}$ for $j=1,2,\dots,p'$, and
\begin{eqnarray*}\omega_{1}:=\prod_{i=1}^{n}\prod_{j=1}^{s_{i}}
\rD(r_{i},k_{r_{i},j})^{l_{r_{i},j}-\delta_{i,n}\delta_{j,s_{n}}}\rD(r_{n}-1,k+k_{r_{n},s_{n}})v_{\varphi}.
\end{eqnarray*}
Again from Lemma \ref{>0-case}, for the parameters in \eqref{Whit-varphi-1}, we know that
\begin{enumerate}
\item[(c1)] there exists $N_{u}\in\Z$ such that
$\{k+k_{r_{n},s_{n}}\}\subseteq\mathcal{T}_{\mathrm{set}}(\omega_{1})
\nsubseteq\bigcup_{m=1}^{p}\mathcal{T}_{\mathrm{set}}(v_{m})$ for all $k>N_{u}$,
\item[(c2)] $\mathrm{lth}_{0}(\omega_{1})\leq1$ and $\mathrm{lth}_{0}(\omega_{1})=1$ if and only if $r_{n}=1$.
Moreover, when $\mathrm{lth}_{0}(\omega_{1})=1$, we have
$\{k+k_{r_{n},s_{n}}\}=\mathcal{T}_{0,\mathrm{set}}(\omega_{1})$.
\end{enumerate}

\vspace{3mm}

Note that $\{0\}\subset\mathcal{D}_{\mathrm{set}}(v_{m})\ \text{for all}\ 1\leq m< m_{0}$.
Now we consider the general case. For any $\nu\in B(\fg_{-},\varphi)$ with $\{0\}\subset\mathcal{D}_{\mathrm{set}}(\nu)$,
we consider
$\big(\rD(-1,k)-\varphi(t^{k})\big).\nu$ for $k>N_{u}$. Set \[\nu=\prod_{i=1}^{Q}\prod_{j=1}^{N_{i}}
\rD(p_{i},m_{p_{i},j})^{q_{p_{i},j}}v_{\varphi}\]
with $0=p_{Q}<\cdots<p_{2}<p_{1}$.
%It follows from Lemma \ref{=0-case} that we only need to focus on the vector
%\begin{eqnarray*}\omega_{2}:&=&\prod_{j=1}^{m-1}\prod_{i=1}^{N_{j}}
%\rD(p_{j},m_{i,p_{j}})^{q_{i,p_{j}}-\delta_{j,m-1}\delta_{i,N_{m-1}}}\rD(p_{m-1}-1,k+m_{N_{m-1},p_{m-1}})\\
%&&\qquad\cdot\prod_{i=1}^{N_{m}}\rD(p_{m},m_{i,p_{m}})^{q_{i,p_{m}}}v_{\psi} \in B^{-}(\chi).\end{eqnarray*}
%We know that
%\begin{itemize}
%\item $\mathrm{lth}_{0}(\omega_{2})\geq1$ for all $k\in\Z$;
%\item if $\mathrm{lth}_{0}(\omega_{2})=1$, then
%$\mathcal{T}_{0,\mathrm{set}}(\omega_{2})\subset\mathcal{T}_{0,\mathrm{set}}(\nu)$ for all $k\in\Z$.
%\end{itemize}
Write
\begin{eqnarray}\label{Whit-varphi-2}
\big(\rD(-1,k)-\varphi(t^{k})\big).\nu=\sum_{l=1}^{q}b_{l}y_{l},\quad k>N_{u},
\end{eqnarray}
where $b_{1},b_{2},\dots,b_{l}\in\C^{\times}$ and $y_{1},y_{2},\dots,y_{l}\in B(\fg_{-},\varphi)$ are distinct elements.
From Lemma \ref{=0-case}, we see that the parameters in \eqref{Whit-varphi-2} satisfy:
\begin{enumerate}
\item[(d1)] if $\mathrm{lth}_{0}(y_l)=0$ for some $1\leq l\leq q$, then
$\mathcal{T}_{\mathrm{set}}(y_{l})\subseteq\mathcal{T}_{\mathrm{set}}(\nu)$;
\item[(d2)] if $\mathrm{lth}_{0}(y_{l})=1$ for some $1\leq l\leq q$, then
$\mathcal{T}_{0,\mathrm{set}}(y_{l})\subseteq\mathcal{T}_{0,\mathrm{set}}(\nu)$.
\end{enumerate}
By compare $(c1)-(c2)$ with $(d1)-(d2)$, it is clear that $y_{l}\neq\omega_{1}$ for all $l=1,\dots,q$.

Thus
\[\big(\rD(-1,k)-\varphi(t^{k})\big).u\neq0\quad\text{for all}\ k>N_{u},\]
which contradicts to the fact that $u$ is a non-zero Whittaker vector. Then we get Claim 1.

\vspace{3mm}

\noindent {\bf Claim 2}
\qquad $\mathcal{D}_{\mathrm{set}}(v_{m})=\{0\}\ \text{for all}\ m=1,\dots,p$.

\noindent Suppose to the contrary that there exists some $1\leq j_{0}\leq p$ such that
$\mathcal{D}_{\mathrm{set}}(v_{j_{0}})\neq\{0\}$. Without loss of generality,
we assume that $\mathcal{D}_{\mathrm{set}}(v_{m})\neq\{0\}$ for $m=1,\dots,j_{0}$ and
$\mathcal{D}_{\mathrm{set}}(v_{j})=\{0\}$ for all $j=j_{0}+1,\dots,p$. We may further assume
that $v_{1}\succ\cdots\succ v_{j_{0}}$.

 Set
\[v_{1}=\prod_{i=1}^{n}\prod_{j=1}^{s_{i}}
\rD(r_{i},k_{r_{i},j})^{l_{r_{i},j}}v_{\varphi}\in B(\fg_{-},\varphi),\]
where $n\geq2$ and $0=r_{n}<\cdots<r_{2}<r_{1}$. It follows from Lemma \ref{=0-case}(2)
that we have
\[\big(\rD(-1,k)-\varphi(t^{k})\big).\sum_{m=1}^{j_{0}}a_{m}v_{m}=c_{0}\omega_{2}+\sum_{j=1}^{P}c_{j}\mu_{j},\quad k\in\Z,\]
where $c_{0}:=a_{1}l_{r_{n-1},s_{n-1}}(r_{n-1}+1)$, $c_{1},\dots,c_{P}\in\C^{\times}$,
\begin{eqnarray*}\omega_{2}&:=&\prod_{i=1}^{n-1}\prod_{j=1}^{s_{i}}
\rD(r_{i},k_{r_{i},j})^{l_{r_{i},j}-\delta_{i,n-1}\delta_{j,s_{n-1}}}\\ &&
\cdot \rD(r_{n-1}-1,k+k_{r_{n-1},s_{n-1}})\prod_{j=1}^{s_{n}}
\rD(r_{n},k_{r_{n},j})^{l_{r_{n},j}}v_{\varphi}\in B(\fg_{-},\varphi),\end{eqnarray*}
and $\mu_{1},\dots,\mu_{P}\in B(\fg_{-},\varphi)$ are distinct elements such that
$\omega_{2}\succ \mu_{j}$ for all $j=1,\dots,P$.
We note that there exists $N_{v_{1}}\in\Z$ such that
$\mathcal{T}_{r_{n-1}-1,\mathrm{set}}(\omega_{2})\not\subset\bigcup_{m=1}^{p}\mathcal{T}_{\mathrm{set}}(v_{m})$
for all $k>N_{v_{1}}$.
But if we write
$\big(\rD(-1,k)-\varphi(t^{k})\big).\sum_{m=j_{0}+1}^{p}a_{m}v_{m}=\sum_{j=1}^{q'}c_{j}'\mu_{j}'$,
where $c_{1}',\dots,c_{q'}'\in\C^{\times}$ and $\mu_{1}',\dots,\mu_{q'}'\in B(\fg_{-},\varphi)$ are distinct elements, we have
\[\mathcal{T}_{\mathrm{set}}(\mu_{j}')\subset\bigcup_{m=1}^{p}\mathcal{T}_{\mathrm{set}}(v_{m})\quad\text{for all}\ j=
1,\dots,q'.\]
This implies that
$\big(\rD(-1,k)-\varphi(t^{k})\big).u\neq0$ for all sufficiently large integers $k$,
a contradiction. This completes the proof.
\end{proof}

In the following, we assume that $\varphi\notin\mathcal{E}$.
We aim to prove that the universal Whittaker module $W(\fg_{-},\varphi)$ is simple, which is implied by
the following result by Lemma \ref{ALZ}.
\begin{proposition}\label{var-not-exp}
If $\varphi\notin\mathcal{E}$, then $W(\fg_{-},\varphi)_{\varphi}=\C v_{\varphi}$.
\end{proposition}
\begin{proof} Suppose to the contrary that $\C v_{\varphi}$ is a proper subset of $W(\fg_{-},\varphi)_{\varphi}$.
Let $u_{+}=\sum_{i=1}^{p}a_{i}v_{i}$ be a vector in $W(\fg_{-},\varphi)_{\varphi}\setminus\C v_{\varphi}$,
where $a_{1},a_{2},\dots,a_{p}\in\C^{\times}$ and $v_{1},v_{2},\dots,v_{p}\in B(\fg_{-},\varphi)$
such that $v_{1}\succ v_{2}\succ \cdots\succ v_{p}$.
It follows from Proposition \ref{reduce-to-0} that we have $\mathrm{lth}_{r}(v_{i})=0$ for $i=1,2,\dots,p$ and
$r\in\Z_{+}$. We write \[v_{1}=\prod_{i=1}^{r}\rD(0,k_{i,1})^{l_{i,1}}v_{\varphi},\]
where $l_{i,1}\in\Z_{+}$, $k_{i,1}\in\Z$ such that $k_{1,1}>k_{2,1}>\cdots>k_{r,1}$,
and $\sum_{i=1}^{r}l_{i,1}\geq 1$.
Then there exists $1\leq s\leq p$ such that $v_{s+1}\prec w$,
for any
\[w\in\Big\{\prod_{i=1}^{r-1}\rD(0,k_{i,1})^{l_{i,1}}\rD(0,k_{r,1})^{l_{r,1}-1}
\rD(0,k)v_{\varphi}\mid k\in\Z\ \text{such that}\ k<k_{r,1}\Big\}.\]
This implies that
$v_{j}=\prod_{i=1}^{r-1}\rD(0,k_{i,1})^{l_{i,1}}\rD(0,k_{r,1})^{l_{r,1}-1}\rD(0,k_{r,j})v_{\varphi}$
for $1\leq j\leq s$, and that $k_{r,1}>k_{r,2}>\cdots>k_{r,s}$.

For $k\in\Z$, we see that
$0=\big(\rD(-1,k)-\varphi(t^{k})\big).u_{+}=
c_{0}w_{0}+\sum_{j=1}^{m}c_{j}w_{j}$,
where \[c_{0}:=l_{r,1}\varphi(t^{k+k_{r,1}})+\varphi(t^{k+k_{r,2}})
+\cdots+\varphi(t^{k+k_{r,s}}),\ c_{1},\dots,c_{m}\in\C\]
and
$w_{0}:=\prod_{i=1}^{r-1}\rD(0,k_{i,1})^{l_{i,1}}
\rD(0,k_{r,1})^{l_{r,1}-1}v_{\varphi},\ w_{1},\dots,w_{m}\in B(\fg_{-},\varphi)$
such that $w_{j}\prec w_{0}$ for $j=1,\dots,m$. This forces
$l_{r,1}\varphi(t^{k+k_{r,1}})+\varphi(t^{k+k_{r,2}})+\cdots+\varphi(t^{k+k_{r,s}})=0$
for all $k\in\Z$, which gives
$\Big(\big(l_{r,1}t^{k_{r,1}}+t^{k_{r,2}}+\dots+t^{k_{r,s}}\big).\varphi\Big)(t^{k})=0$.
Thus \[\big(l_{r,1}t^{k_{r,1}}+t^{k_{r,2}}+\dots+t^{k_{r,s}}\big).\varphi=0,\]
i.e., $\varphi\in\mathcal{E}$, which is a contradiction. This completes the proof.
\end{proof}

Similar to the proof of Proposition \ref{Whittaker-vector-N}, we have
\begin{proposition}
Let $\varphi\in\mathcal{E}$. For any $f_{1}(t),\dots, f_{s}(t)\in\mathrm{Ann}(\varphi)$, we have
\[\prod_{i=1}^{s}\rD(0,f_{i})v_{\varphi}\in W(\fg_{-},\varphi)_{\varphi}.\]
\end{proposition}

\begin{remark}Suppose $\varphi\in\mathcal{E}$.
It is an interesting problem that if
\[\prod_{i=1}^{s}\rD(0,f_{i})v_{\varphi},\quad f_{1}(t),\dots, f_{s}(t)\in\mathrm{Ann}(\varphi),\]
exhaust all the Whittaker vectors in $W(\fg_{-},\varphi)$.
\end{remark}

\section{Whittaker modules over the loop Virasoro algebra}
In this section, we give the criterion for the simplicity of the universal Whittaker modules over loop Virasoro algebra.
Firstly, we recall the definition of loop Virasoro algebra from \cite{GLZ}.

Let $Vir$ be the $\textsl{Virasoro algebra}$,
it with basis $\{c,\rd_{i}\mid i\in\Z\}$ and the bracket (for $i, j\in\Z$):
\[[\rd_{i},\rd_{j}]=(j-i)\rd_{i+j}+\delta_{i,-j}\frac{i^{3}-i}{12}\bm{c};\quad [\rd_{i},\bm{c}]=0.\]
It is clear that $\mathcal{W}$ is a subalgebra of $Vir$.

The $\textsl{loop Virasoro alegbar}$ $\mathfrak{L}$ is the Lie algebra that is the tensor product of
the Virasoro Lie algebra $Vir$ and the Laurent polynomial algebra $\mathcal{A}$,
i.e., $\mathfrak{L}=Vir\otimes\mathcal{A}$ subject to the commutator relation:
\[[\rd_{i}\otimes t^{k},\rd_{j}\otimes t^{l}]=(j-i)\rd_{i+j}\otimes t^{k+l}
+\delta_{i,-j}\frac{i^{3}-i}{12}\bm{c}\otimes t^{k+l},\]
\[[\rd_{i}\otimes t^{k},\bm{c}\otimes t^{l}]=0\]
for $i,j,k,l\in\Z$.

Recall that $N\in\Z_{+}$ is a fixed positive integer.
Let \[\mathfrak{L}_{\geq N}=\text{Span}_{\C}\{\rd_{i}\otimes t^{k}, \bm{c}\otimes t^{k}\mid i\geq N,\ k\in\Z\}.\]
We know that $(\mathfrak{L},\mathfrak{L}_{\geq N})$ is a Whittaker pair.
Similarly, for any Whittaker function $\phi:\mathfrak{L}_{\geq N}\rightarrow\C$, we can define the
universal Whittaker module $W(\mathfrak{L}_{\geq N},\phi)$ over $\mathfrak{L}$ as follows:
\begin{align}\label{Vir-Whit}
W(\mathfrak{L}_{\geq N},\phi)=\U(\mathfrak{L})\otimes_{\U(\mathfrak{L}_{\geq N})}\C v_{\phi},
\end{align}
where $\C v_{\phi}$ is the one dimensional $\mathfrak{L}_{\geq N}$-module given by
$x.v_{\phi}=\phi(x)v_{\phi}$ for any $x\in\mathfrak{L}_{\geq N}$.

\begin{remark}
Set $\mathfrak{L}_{\leq-N}=\text{Span}_{\C}\{\rd_{i}\otimes t^{k}, \bm{c}\otimes t^{k}\mid i,k\in\Z, i\leq-N\}$.
It is clear that $(\mathfrak{L},\mathfrak{L}_{\leq-N})$ is a Whittaker pair. One can
define the universal Whittaker modules associated to $(\mathfrak{L},\mathfrak{L}_{\leq-N})$ similar to that of
\eqref{Vir-Whit}. Note that there exists an involution of $Vir$
given by $\rd_{n}\mapsto-\rd_{-n}$, $n\in\Z$, $\bm{c}\mapsto-\bm{c}$.
Via the involution of $Vir$ (hence of $\mathfrak{L}$), it is
enough to consider the Whittaker module $W(\mathfrak{L}_{\geq N},\phi)$ for $\mathfrak{L}$ defined in \eqref{Vir-Whit}.
However, for the loop Witt algebra $\fg$, we have no such involution. Thus, we have to consider
the Whittaker modules $W(\fg_{\geq N},\phi)$ and $W(\fg_{-1},\phi)$ for $\fg$.
\end{remark}

For convenience, write $\rD(n,k)=\rd_{n}\otimes t^{k}$, $n,k\in\Z$. Set
\begin{eqnarray*}&&B(\mathfrak{L}_{\geq N},\phi)=\{\rD(r_{1},k_{r_{1},1})\rD(r_{1},k_{r_{1},2})\cdots
\rD(r_{1},k_{r_{1},s_{1}})\rD(r_{2},k_{r_{2},1})\cdots
\rD(r_{2},k_{r_{2},s_{2}})\\&&
\quad\cdots\cdots\rD(r_{n},k_{r_{n},1})\rD(r_{n},k_{r_{n},2})\cdots
\rD(r_{n},k_{r_{n},s_{n}})v_{\phi}\mid n\in\N,\ k_{r_{i},j}\in\Z\\&&
r_{n}<\cdots<r_{2}<r_{1}\leq N-1,\ k_{r_{i},s_{i}}\leq\cdots\leq k_{r_{i},2}\leq k_{r_{i},1},\
1\leq i\leq n,\ 1\leq j \leq s_{i}\}.\end{eqnarray*}
By Poincar-Birkhok-Witt theorem, we know that $B(\mathfrak{L}_{\geq N},\phi)$ forms a basis of $W(\mathfrak{L}_{\geq N},\phi)$.

For a vector \begin{eqnarray}\label{Vir-case}u=\prod_{i=1}^{n}\prod_{j=1}^{s_{i}}
\rD(r_{i},k_{r_{i},j})^{l_{r_{i},j}}v_{\phi}\in B(\mathfrak{L}_{\geq N},\phi),\end{eqnarray}
where $r_{n}<\cdots<r_{2}<r_{1}\leq N-1$, $l_{r_{i},j}\in\Z_{+}$, and $k_{r_{i},s_{i}}\leq\cdots\leq k_{r_{i},2}\leq k_{r_{i},1}$,
we let \[\mathrm{lth}_{r_{i}}(u)=\sum_{j=1}^{s_{i}}l_{r_{i},j},\quad
\mathrm{lth}(u)=\sum_{i=1}^{n}\sum_{j=1}^{s_{i}}l_{r_{i},j},\]
\[\mathcal{D}(u)=\big(\underbrace{r_{1},r_{1},\dots,r_{1}}_{\mathrm{lth}_{r_{1}}(u)-\mathrm{times}},
\underbrace{r_{2},r_{2},\dots,r_{2}}_{\mathrm{lth}_{r_{2}}(u)-\mathrm{times}},\dots,
\underbrace{r_{n},r_{n},\dots,r_{n}}_{\mathrm{lth}_{r_{n}}(u)-\mathrm{times}}\big),\]
$$\begin{aligned}\mathcal{T}(u)=\big(\underbrace{k_{r_{1},1},k_{r_{1},1},\dots,k_{r_{1},1}}_{l_{r_{1},1}-\mathrm{times}},
\underbrace{k_{r_{1},2},k_{r_{1},2},\dots,k_{r_{1},2}}_{l_{r_{1},2}-\mathrm{times}},\dots,
\underbrace{k_{r_{1},s_{1}},k_{r_{1},s_{1}},\dots,k_{r_{1},s_{1}}}_{l_{r_{1},s_{1}}-\mathrm{times}},\\ \cdots,
\underbrace{k_{r_{n},1},k_{r_{n},1},\dots,k_{r_{n},1}}_{l_{r_{n},1}-\mathrm{times}},\cdots,
\underbrace{k_{r_{n},s_{n}},k_{r_{n},s_{n}},\dots,k_{r_{n},s_{n}}}_{l_{r_{n},s_{n}}-\mathrm{times}}\big),\end{aligned}$$
\[\mathcal{D}_{\mathrm{set}}(u)=\{r_{1},\dots,r_{n}\},\quad\mathcal{T}_{r_{i},
\mathrm{set}}(u)=\{k_{r_{i},j}\mid j=1,2,\dots,s_{i}\}\ \text{for}\ i=1,2,\dots,n,\]and
\[\mathcal{T}_{\mathrm{set}}(u)=\{k_{r_{i},j}\mid i=1,2,\dots,n,\ j=1,2,\dots,s_{i}\}.\]

For any $r\in\Z_{+}$, recall the total order on $\Z^{r}$ defined in \eqref{total-order}.
Now we define a principle total order ``$\succ$'' on $B(\mathfrak{L}_{\geq N},\phi)$ as follows:
for different $u,v\in B(\mathfrak{L}_{\geq N},\phi)$, set $u\succ v$ if and only if one of the following conditions satisfy:
     \begin{itemize}
     \item $\text{lth}(u)>\text{lth}(v)$;
     \item $\text{lth}(u)=\text{lth}(v)$ and $\mathcal{D}(u)\succ\mathcal{D}(v)$
     under the order $\succ$ on $\Z^{\text{lth}(u)}$;
     \item $\text{lth}(u)=\text{lth}(v)$, $\mathcal{D}(u)=\mathcal{D}(v)$,
     and $\mathcal{T}(u)\succ\mathcal{T}(v)$
     under the order $\succ$ on $\Z^{\text{lth}(u)}$.
     \end{itemize}

Let $\phi:\mathfrak{L}_{\geq N}\rightarrow\C$ be a Whittaker function.
For any $n\geq N$, define $\phi_{n}\in\mathcal{A}^*$ such that
\begin{align}\phi_{n}(t^{r})=\phi(\rd_{n}\otimes t^{r})\quad \text{for all}\quad r\in\Z.\end{align}
Since $\phi$ is a Lie algebra homomorphism, we know that
\[\phi_{i}=0\quad\text{for all}\quad i\geq2N+1.\]

\begin{proposition}
If $\phi_{2N-1},\phi_{2N}\in\mathcal{E}$, then $W(\mathfrak{L}_{\geq N},\phi)$ is reducible.
\end{proposition}
\begin{proof}
The proof is similar to that of Proposition \ref{psi-reducible}. Just note that we can get a
Whittaker vector $u$ in $\U(\rd_{N-1}\otimes\mathcal{A})v_{\phi}\setminus\C v_{\phi}$,
which generates a proper submodule $\U(\mathfrak{L})u$ of $W(\mathfrak{L}_{\geq N},\phi)$.
\end{proof}

\begin{proposition}
Let $v=\sum_{i=1}^{p}a_{i}v_{i}\in W(\mathfrak{L}_{\geq N},\phi)_{\phi}\setminus\C v_{\phi}$,
where $a_{1},a_{2},\dots,a_{p}\in\C^{\times}$ and $v_{1},v_{2},\dots,v_{p}\in B(\mathfrak{L}_{\geq N},\phi)$
are distinct elements. Then for any $j<0$, we have
\[\mathrm{lth}_{j}(v_{i})=0\quad \text{for all}\quad i=1,2,\dots,p.\]
\end{proposition}
\begin{proof}
Suppose to the contrary we have \[I:=\{1\leq i\leq p\mid\sum_{j<0,\ j\in\Z}\mathrm{lth}_{j}(v_{i})\geq1\}\neq\emptyset.\]
Without loss of generality, we assume that $I=\{1,2,\dots,i_{0}\}$
for some $1\leq i_{0}\leq p$ and $v_{1}\succ\cdots\succ v_{i_{0}}$.
Let $v_{1}$ be such as in \eqref{Vir-case}. Then choose
$1\leq m_{0}\leq n$ such that $r_{m_0-1}> N+r_{n}\geq r_{m_0}$. For $1\leq i\leq n$, write
$x_{i}=\prod_{j=1}^{s_{i}}
\rD(r_{i},k_{r_{i},j})^{l_{r_{i},j}}$. That is $v_{1}=x_{1}x_{2}\cdots x_{n}v_{\phi}$.

\vspace{2mm}

\noindent {\bf Claim}. $\big(\rD(N,k)-\phi_{N}(t^{k})\big).v\neq0$ for all sufficiently large integers $k$.

It is straightforward to see that
\[\big(\rD(N,k)-\phi_{N}(t^{k})\big).\sum_{i=1}^{p}a_{i}v_{i}=
b_{0}w_{0}+\sum_{j=1}^{q}b_{j}w_{j},\]
where $b_{0}:=l_{s_{n},r_{n}}(r_{n}-N),b_{1},\dots,b_{q}\in\C^{\times}$ and
$$\begin{aligned}w_{0}:=x_{1}\cdots x_{m_{0}-1}\rD(N+r_{n},k+k_{s_{n},r_{n}})x_{m_{0}}\cdots x_{n-1}
\prod_{j=1}^{s_{n}}\rD(r_{n},k_{r_{n},j})^{l_{r_{n},j}-\delta_{j,s_{n}}}v_{\phi},\\ w_{1},\dots,w_{q}\in B(\mathfrak{L}_{\geq N},\phi)
\end{aligned}$$
such that $w_{j}\neq w_{0}$ for $j=1,\dots, q$. In fact, we can deduce that, for any $2\leq l\leq p$,
if write $\big(\rD(N,k)-\phi_{N}(t^{k})\big).v_{l}=\sum_{r=1}^{Q}c_{r}u_{r}$, where $c_{1},\dots,c_{Q}\in\C^{\times}$
and $u_{1},\dots,u_{Q}\in B(\mathfrak{L}_{\geq N},\phi)$ are distinct elements, then we have
$u_{r}\neq w_{0}$ for $r=1,\dots,Q$.
Then we know that the claim holds. This contradicts the fact that $v\in W(\mathfrak{L}_{\geq N},\phi)_{\phi}$.
We complete the proof.
\end{proof}
From Propositions \ref{EXP-2N} and \ref{EXP-2N-1}, we have
\begin{corollary}
(1) If $\phi_{2N}\notin\mathcal{E}$, then $W(\mathfrak{L}_{\geq N},\phi)_{\phi}=\C v_{\phi}$.\\
(2) If $\phi_{2N-1}\notin\mathcal{E}$ and $\phi_{2N}\in\mathcal{E}$, then $W(\mathfrak{L}_{\geq N},\phi)_{\phi}=\C v_{\phi}$.
\end{corollary}
Thus from Lemmas \ref{ALZ} and \ref{CJ}, we have
\begin{thm}\label{Vir-Whittaker-module}Let $\phi:\mathfrak{L}_{\geq N}\rightarrow\C$ be a Whittaker function.
Then the universal Whittaker $\mathfrak{L}$-module $W(\mathfrak{L}_{\geq N},\phi)$ is simple if and only if either $\phi_{2N}\notin\mathcal{E}$ or
$\phi_{2N-1}\notin\mathcal{E}$.
\end{thm}
From Proposition \ref{Whittaker-vector-N}, we have
\begin{corollary}
Let $\phi:\mathfrak{L}_{\geq N}\rightarrow\C$ be a Whittaker function
such that $\phi_{2N-1},\phi_{2N}\in\mathcal{E}$. For any
$f_{1}(t),\dots, f_{s}(t)\in\mathrm{Ann}(\phi_{2N-1})\cap\mathrm{Ann}(\phi_{2N})$, we have
\[\prod_{j=1}^{s}\rD(N-1,f_{j})v_{\phi}\in W(\mathfrak{L}_{\geq N},\phi)_{\phi}.\]
\end{corollary}

\begin{remark}
Let $Vir_{\geq N}=\bigoplus_{i\geq N}\C\rd_{i}+\C\bm{c}$ and
let $\phi:Vir_{\geq N}\rightarrow\C$ be a Lie algebra homomorphism, where $\phi(\bm{c})=\dot{z}$. Then the Whittaker module
$L_{\phi,\dot{z}}$ over $Vir$ is simple if and only if $\phi(\rd_{2N})\neq0$ or $\phi(\rd_{2N-1})\neq0$.
For details, we refer the reader to see \cite[Theorem 7]{LGZ}.
This implies that our results may be regarded as a generalization of the coordinated algebra
from complex field $\C$ to the Laurent polynomial ring $\mathcal{A}$.
\end{remark}

As an aplication, we discuss the Whittaker module over the affine Lie algebra $\widehat{\mathfrak{sl}_{2}}$.
Let $\mathfrak{sl}_{2}$ be the Lie algebra of traceless $2\times2$-matrices over $\C$. We fix a standard basis
$\{\re,\rh,\rf\}$ of $\mathfrak{sl}_{2}$ such that $[\re,\rf]=\rh$, $[\rh,\re]=2\re$, and $[\rh,\rf]=-2\rf$.
Let $(\cdot,\cdot)$ be a nondegenerate symmetric bilinear form on $\mathfrak{sl}_{2}$.
The affine Lie algebra $\widehat{\mathfrak{sl}_{2}}$ associated to $\mathfrak{sl}_{2}$ is defined as
$\mathfrak{sl}_{2}\otimes\C[t,t^{-1}]\oplus\C\bm{k}$, where $\bm{k}$ is the canonical central element
and the Lie algebra structure is given by
\[[x\otimes t^{k},y\otimes t^{l}]=[x,y]\otimes t^{k+l}+k(x,y)\delta_{k+l,0}\bm{k}\]
for any $x,y\in\mathfrak{sl}_{2}$ and $k,l\in\Z$.
It is clear that $(\widehat{\mathfrak{sl}_{2}},\re\otimes\C[t,t^{-1}]\oplus\C\bm{k})$ is a Whittaker pair.
Then for any $\phi\in\mathcal{A}^*$ and $\dot{k}\in\C$,
we define the universal Whittaker $\widehat{\mathfrak{sl}_{2}}$-module
\[W(\widehat{\mathfrak{sl}_{2}},\phi,\dot{k})=\U(\widehat{\mathfrak{sl}_{2}})
\otimes_{\U(\re\otimes\C[t,t^{-1}]\oplus\C\bm{k})}\C w_{\phi},\]
where $\C w_{\phi}$ is the one dimensional $(\re\otimes\C[t,t^{-1}]\oplus\C\bm{k})$-module determined
by $(\re\otimes t^{k}).w_{\phi}=\phi(t^{k})w_{\phi}$ and $\bm{k}.w_{\phi}=\dot{k}w_{\phi}$, $k\in\Z$.
Recall that \[W(\widehat{\mathfrak{sl}_{2}},\phi,\dot{k})_{\phi}
=\{v\in W(\widehat{\mathfrak{sl}_{2}},\phi,\dot{k})\mid \big(\re\otimes t^{k}-\phi(t^{k})\big).v=0\ \text{for all}\ k\in\Z\}.\]
From Proposition \ref{some-property-Whit}, we have
\begin{proposition}
Let $\phi\in\mathcal{A}^*$ and $\dot{k}\in\C$. Then
\[W(\widehat{\mathfrak{sl}_{2}},\phi,\dot{k})_{\phi}\subseteq\U(\rh\otimes\C[t,t^{-1}])w_{\phi}.\]
\end{proposition}

Similar to the proof of Proposition \ref{E-Reducible} and Proposition \ref{var-not-exp},
we can obtain the following result.
\begin{proposition}\label{sl-2}
Let $\phi\in\mathcal{A}^*$ and $\dot{k}\in\C$. Then the universal Whittaker $\widehat{\mathfrak{sl}_{2}}$-module $W(\widehat{\mathfrak{sl}_{2}},\phi,\dot{k})$ is simple if and only if $\phi\in\mathcal{E}$.
\end{proposition}

\begin{remark} Let $\widetilde{\mathfrak{sl}_{2}}=\widehat{\mathfrak{sl}_{2}}\oplus\C\bm{\rd}$
be the affine Kac-Moody Lie algebra of type $A_{1}^{(1)}$. Note that
$(\widetilde{\mathfrak{sl}_{2}},\re\otimes\C[t,t^{-1}]\oplus\C\bm{k})$ is a Whittaker pair. Then
for any $\phi\in\mathcal{A}^*$ and $\dot{k}\in\C$, we can similarly
define the universal Whittaker $\widetilde{\mathfrak{sl}_{2}}$-module
\[W(\widetilde{\mathfrak{sl}_{2}},\phi,\dot{k})=\U(\widetilde{\mathfrak{sl}_{2}})
\otimes_{\U(\re\otimes\C[t,t^{-1}]\oplus\C\bm{k})}\C w_{\phi}.\]
In \cite{M}, Mazorchuk gave a sufficient condition for the Whittaker module
$W(\widetilde{\mathfrak{sl}_{2}},\phi,\dot{k})$ to be simple.
While we can prove that $W(\widetilde{\mathfrak{sl}_{2}},\phi,\dot{k})$ is simple if and only if
$\phi\in\mathcal{E}$ by using Proposition \ref{sl-2}. Our condition is more general than \cite{M}.
\end{remark}

\end{document}